\documentclass[12pt]{article}
\usepackage{fancyvrb}

\usepackage[tikz]{bclogo}
\newenvironment{code}
{\VerbatimEnvironment
  \begin{bclogo}[couleur = black!10, couleurBord = black!20,
  arrondi = 0.1, logo = \bchorloge]
  {\small{Maple Program}}\par\nobreak\smallskip
  \begin{Verbatim}}
{\end{Verbatim}\end{bclogo}}
\usepackage{tcolorbox}
\usepackage{hyperref}
\usepackage{courier}
\usepackage[top=1in,left=1in]{geometry}
\usepackage{amsmath,amsthm,amssymb}
\usepackage{enumerate, xcolor}
\usepackage{enumitem}

\newtheorem{thm}{Theorem}

\newtheorem{prop}[thm]{Proposition}
\newtheorem{cor}[thm]{Corollary}

\theoremstyle{definition}
\newtheorem*{defi}{Definition}
\newtheorem*{rem}{Remark}

\begin{document}

\title{Ansatz in a Nutshell \\
\large A comprehensive step-by-step guide to\\
polynomial, $C$-finite, holonomic, and $C^2$-finite sequences}

\author{Tipaluck Krityakierne$^1$ \and Thotsaporn Aek Thanatipanonda$^{2,*}$}
\date{%
    \footnotesize{$^1$Department of Mathematics, Faculty of Science, Mahidol University, Bangkok, Thailand\\%
    $^2$Science Division, Mahidol University International College, Nakhon Pathom, Thailand}\\
    $^*$Correspondence: \href{mailto:thotsaporn@gmail.com}{\texttt{thotsaporn@gmail.com}}\\[2ex]%
}

\maketitle

\begin{abstract}
Given a sequence 1, 1, 5, 23, 135, 925, 7285, 64755, 641075, 6993545, 83339745,..., how can we guess a formula for it? 
This article will quickly walk you through the concept of {\it ansatz} for classes of polynomial, $C$-finite, holonomic, and the most recent addition $C^2$-finite sequences.
For each of these classes, we discuss in detail various aspects of the guess and check, generating functions, closure properties, and closed-form solutions. Every theorem is presented with an accessible proof, followed by several examples intended to motivate the development of the theories. Each example is accompanied by a Maple program with the purpose of demonstrating use of the program in solving problems in this area. 
While this work aims to give a comprehensive review of existing ansatzes, we also systematically fill a research gap in the literature by providing theoretical and numerical results for the $C^2$-finite sequences.  
We hope the readers will enjoy the journey through our unifying framework for the study of ansatzes.
\end{abstract}


\section{Getting started}
When we come across a word or a phrase we have never seen before, we look it up in a dictionary. Likewise, whenever we encounter a sequence for which we do not know a formula, we could look it up in the Sloane's
OEIS, an online dictionary for number sequences \cite{OEIS}. 
However, as it is not possible that the OEIS database consists of everyone of them, wouldn't it be great if we could find a formula by ourselves, regardless of whether or not our sequence is there? And this is precisely the central theme of this work.

\begin{tcolorbox}
\textbf{Theme:}   {Given a sequence $a_n, \;\ n=0,1,2,\dots$, find or guess a (homogeneous) linear recurrence relation of the form:
\[  c_r(n)a_{n+r}+  c_{r-1}(n)a_{n+r-1}+ \dots + c_0(n)a_{n} = 0,  \]
where $c_i(n)$ could be a constant, a polynomial in $n$,
or even a linear recurrence (in $n$) itself.}
\end{tcolorbox}

\begin{figure}
\label{fig:paper_structure}
    \centering
    \includegraphics[scale=0.29]{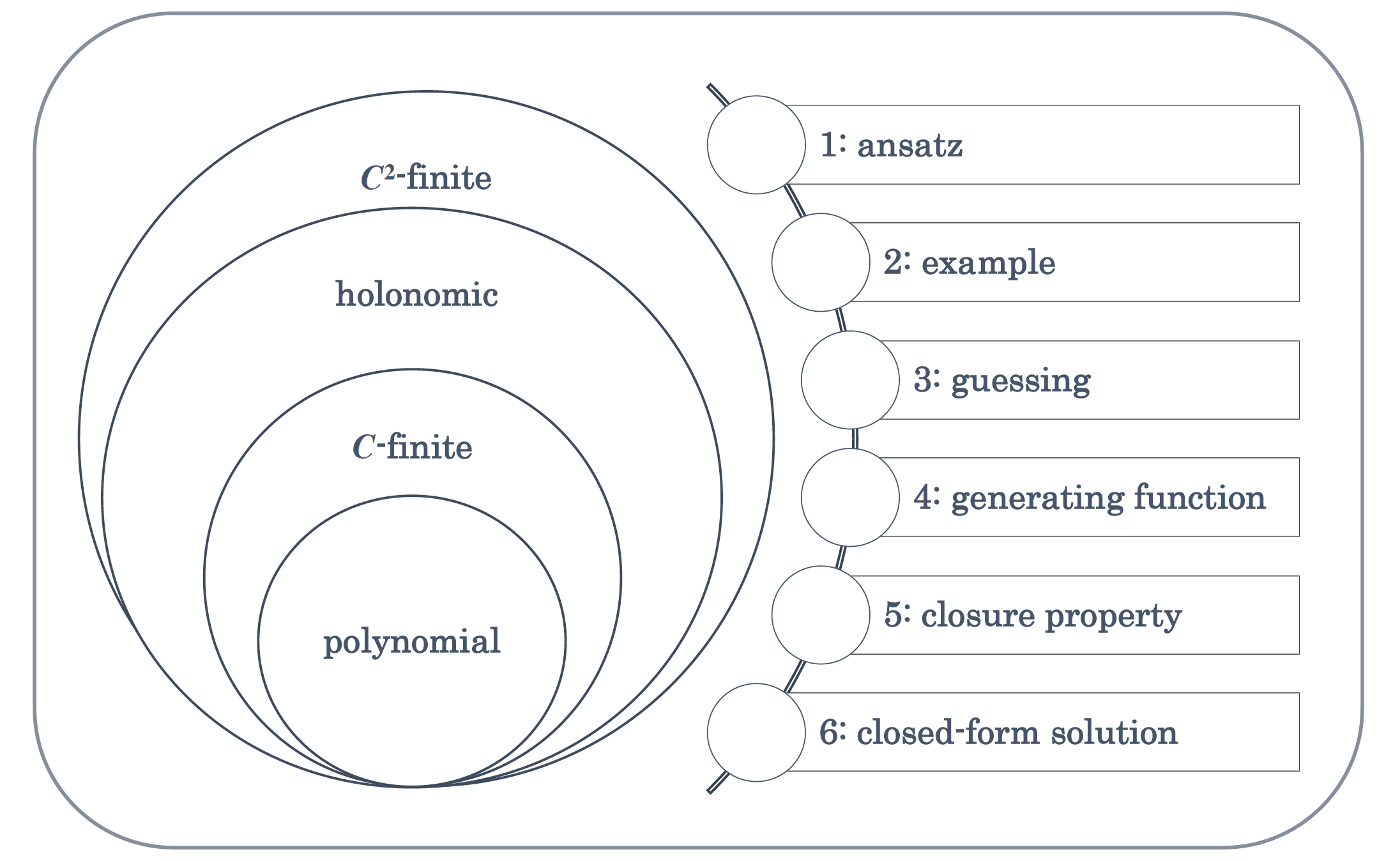}
    \caption{Left: the classes of sequences in this study; Right: the structure of the paper}
\end{figure}

This paper is intended to provide a comprehensive background on the concept of ansatz for classes of polynomial, $C$-finite, holonomic, and the recently developed $C^2$-finite sequences. The inclusion relation of the four classes is depicted in Figure \ref{fig:paper_structure}. 
While there are other classes of sequences (e.g. algebraic sequences, hypergeometric sequences), we focus exclusively on these classes so that we can provide an in-depth review of the subject from both the theoretical and practical aspects. In addition, the reader may notice beauty in the embedding nature of these sequences, i.e. a polynomial sequence as a coefficient of the holonomic sequences, and a $C$-finite sequence as a coefficient of the $C^2$-finite sequences. The right panel of Figure \ref{fig:paper_structure} gives a quick overview of the schematic structure of the paper. 

Consequently, the aim of this paper is twofold: firstly, to present a step-by-step guide to ansatzes, and secondly to extend the knowledge of existing ansatzes to the class of $C^2$-finite sequences in a systematic manner. We have tried to keep the paper as self-contained yet easy-to-follow as possible. Every theorem is presented with a proof, where much effort has been made to make the proof accessible and transparent. Every theorem is illustrated with examples or applications that stimulated the development of the concept. While most of these examples are solved analytically, Maple has been used in several intermediate calculations and simplification of algebraic expressions.  At the end of each example, we demonstrate the command required to enter the data of the problem into our Maple program.  
 
Last but not least, it is worth emphasizing that the steps we provided in the proofs of theorems can be used to construct problem solutions, and in fact we have followed these steps precisely when implementing our Maple program.  The interested reader is invited to study the codes accompanying this paper, provided at \texttt{\url{https://thotsaporn.com/Ansatz.html}} or implement the program in their favorite programming environment. 

The outline of this paper is as follows. In the next section, we give a comprehensive review of the existing ansatz along with important theoretical properties.
Section \ref{sec:C2} presents results related to $C^2$-finite sequences where we formally give a definition of $C^2$-finite, and establish several results concerning the generating function and the closure properties.  We also list a few interesting unsolved problems in that section. 
Through our presentation, we hope that this paper will provide a unifying framework for understanding the various aspects of ansatz in the classes of polynomial, $C$-finite, holonomic, and $C^2$-finite sequences. 
\section{Comprehensive review of existing ansatz}
This section concentrates on a comprehensive review of ansatzes in the classes of polynomials (as a sequence), $C$-finite, and holonomic sequences. For a list of excellent resources, see for example, \cite{KZ,Z1} for $C$-finite, \cite{K} for holonomic, as well as \cite{BP,KP} and the bibliography. There are also several available software tools that provide excellent computing facilities for ansatzes.  The reader is encouraged to check out e.g. the Maple package \texttt{gfun} \cite{SZ} and Mathematica package \texttt{GeneratingFunctions} \cite{M} which provide commands for dealing with holonomic sequences. Although this list of references is by no means exhaustive, we acknowledge the contributions of all pioneer researchers in the field.

\subsection{Polynomial as a sequence}
\begin{enumerate}[leftmargin=0.4in,label=\textbf{\arabic*.}]
\item {\bf Ansatz:} $a_n = c_kn^k+c_{k-1}n^{k-1}+\dots+c_0.$ 

\item {\bf Example:} Let $ \displaystyle a_n = \sum_{i=1}^n i^2.$ \;\
Here, $a_n = \dfrac{n^3}{3}+\dfrac{n^2}{2}+\dfrac{n}{6} = \dfrac{n(n+1)(2n+1)}{6}$.

\newpage
\item {\bf Guessing:} 

Input: the degree $k$ of polynomial and 
sequence $a_n$ (of length more than $k+1$).
We solve the system of linear equations for $c_0, c_1, \dots, c_k$:

\[  \begin{bmatrix}
c_0 & c_1 & \dots & c_k
\end{bmatrix} 
\begin{bmatrix}
1 & 1 & 1 & \dots & 1\\
0 & 1 & 2 & \dots & k\\
& & \dots \\
0^k & 1^k & 2^k & \dots & k^k\\
\end{bmatrix} =
\begin{bmatrix}
a_0 & a_1 & \dots & a_k
\end{bmatrix}.
\]

This matrix equation can be solved quickly using the inverse of the (Vandermonde) matrix that it is multiplied with.
Once all these $c_i$'s are obtained, we have to verify that our conjectured polynomial is valid for the rest of the terms $a_n,  n > k$ in the sequence (otherwise there is no solution).


\begin{code}
> A:= [seq(add(i^2,i=0..n),n=0..20)]; 
> GuessPol(A,0,n);
\end{code}



\item  {\bf   Linear recurrence:} 

The following proposition provides a linear recurrence relation for a polynomial sequence.
For convenience, we define the \emph{left shift operator} on $a_n$ by \;\ $Na_n = a_{n+1}.$   

\begin{prop} \label{prop1}
$a_n$ is a polynomial in $n$ of degree at most $k$ 
if and only if $(N-1)^{k+1}a_n = 0$. 
\end{prop}

\begin{proof}
Assume $a_n$ is a polynomial in $n$ of degree $k$.
Observe that  $(N-1)a_n = a_{n+1}-a_n$ 
is a polynomial whose degree is reduced by (at least) 1. 

By applying the operator $(N-1)$ repeatedly $k+1$ times, we obtain 
the homogeneous linear recurrence relation of $a_n$, i.e.
\[  (N-1)^{k+1}a_n = 0.  \]
This completes the proof of the forward direction. For the reverse direction, suppose the condition $(N-1)^{k+1}a_n = 0$ holds. 
Consider 
\[ a_n = N^n a_0 = [1+(N-1)]^n a_0 = \sum_{i=0}^n \binom{n}{i}(N-1)^ia_0 .   \]
It follows from the assumption that 
\begin{equation}
\label{eq:an_polynomial}
a_n = \sum_{i=0}^k \binom{n}{i}(N-1)^ia_0.
\end{equation}
Hence, $a_n$ is a polynomial (in $n$) of degree at most $k.$
\end{proof} 

We obtain the following corollary immediately from \eqref{eq:an_polynomial}. Note that this has also been discussed in Proposition 1.4.2 of \cite{S2}.
\begin{cor}
A polynomial $a_n$ of degree at most $k$ can be written in expanded form as
\[
a_n = \sum_{i=0}^k \binom{n}{i}(N-1)^ia_0.
\]
\end{cor}


{\bf Example:}  Let  $a_n= \dfrac{n(n+1)(2n+1)}{6}.$
Then, $a_n$ can be written as $2\binom{n}{3}+3\binom{n}{2}+\binom{n}{1}.$ 

An important consequence of this observation is the following.
\begin{prop} \label{prop2}
For a non-negative integer $k$,
$\displaystyle \sum_{i=0}^n i^k$ is a polynomial of degree at most $k+1.$
\end{prop} 

\begin{proof}
This is easy to see. Let $\displaystyle a_n = \sum_{i=0}^n i^k$.
Then, $(N-1)a_n = a_{n+1}-a_n = (n+1)^k,$ a polynomial of degree $k.$
Hence, from Proposition \ref{prop1}, $(N-1)^{k+2}a_n=0$ and, by reverse observation 
of Proposition \ref{prop1}, $a_n$ is a polynomial of degree at most $k+1.$ 
\end{proof} 

\begin{rem} Proposition \ref{prop2} is very useful as knowing such a bound on the polynomial degree allows us to make guesses rigorous. In particular, finding a polynomial equation for $\displaystyle a_n = \sum_{i=0}^n i^k$, for some fixed $k$, amounts to fitting a polynomial of degree $k+1$ to a set of data points $a_n, \, n=0,1,2,\dots, k+1.$ 
\end{rem}

\item  {\bf   Generating function:} 

Every sequence corresponds to a generating function that comes in handy when determining the formula of the sequence, as we shall see later. Let us note that the generating function considered here is a \emph{formal} power series in the sense that it is regarded as an algebraic object, thereby ignoring the issue of convergence. The next proposition establishes a connection between polynomial sequences and the generating functions.

\begin{prop}
\label{prop:generating_polynomial}
Let $\displaystyle f(x) = \sum_{n=0}^{\infty} a_nx^n $
where $a_n$ is a polynomial in $n$ of degree $k$. 
Then,
\begin{equation}
\label{eq:f_polynomial}
    f(x) = \dfrac{P(x)}{(1-x)^{k+1}}, 
\end{equation}  
for some polynomial $P(x)$ of degree at most $k$.
\end{prop}

\begin{proof}
Assume $a_n$ is a polynomial in $n$ of degree $k$. Then,
\begin{align*}
(1-x)^{k+1}f(x) &= \sum_{n=0}^{\infty} a_nx^n(1-x)^{k+1}\\
&=  \sum_{n=0}^{\infty} \sum_{i=0}^{k+1} \binom{k+1}{i}(-1)^i a_nx^{n+i}\\
&= \sum_{i=0}^{k+1} \sum_{n=i}^{\infty} \binom{k+1}{i}(-1)^i a_{n-i}x^{n}\\
&=  \sum_{n=k+1}^{\infty} \sum_{i=0}^{k+1} \binom{k+1}{i}(-1)^i a_{n-i}x^{n}
+  \sum_{n=0}^{k} \sum_{i=0}^{n} \binom{k+1}{i}(-1)^i a_{n-i}x^{n}.\\
\end{align*}

The first summation is essentially zero as, for each $n \geq k+1$, 
\[
\displaystyle \sum_{i=0}^{k+1} \binom{k+1}{i}(-1)^i a_{n-i} = (N-1)^{k+1}a_{n-k-1} = 0,
\]
by Proposition \ref{prop1}. Hence,
\[ (1-x)^{k+1}f(x) = P(x), \]
where 
\begin{equation}
\label{eq:generating_polynomial}
   \displaystyle P(x) = \sum_{n=0}^{k} \sum_{i=0}^{n} \binom{k+1}{i}(-1)^i a_{n-i}x^{n}, 
\end{equation}
a polynomial of degree at most $k$.
\end{proof}

{\bf Example:} The generating function $f(x)$ for $\displaystyle a_n = \sum_{i=1}^n i^2$ is
$\dfrac{x^2+x}{(1-x)^4}.$

\begin{code}
> GenPol(n*(n+1)*(2*n+1)/6,n,x);
\end{code}



The next proposition states the converse of Proposition \ref{prop:generating_polynomial}.

\begin{prop}
\label{prop:converse_polynomial}
Assume that $\displaystyle f(x) = \sum_{n=0}^{\infty} a_nx^n $ satisfies
\[
    f(x) = \dfrac{P(x)}{(1-x)^{k+1}}, 
\]
for some polynomial $P(x)$ of degree $k$.
Then, $a_n$ is a polynomial sequence of degree at most $k$.
\end{prop}

\begin{proof}
Since 
\[
\displaystyle  (1-x)^{k+1}f(x) =  \sum_{n=k+1}^{\infty} \sum_{i=0}^{k+1} \binom{k+1}{i}(-1)^i a_{n-i}x^{n}
+  \sum_{n=0}^{k} \sum_{i=0}^{n} \binom{k+1}{i}(-1)^i a_{n-i}x^{n}
\]
is a polynomial of degree $k$, the first summation must be zero implying $(N-1)^{k+1}a_{n-k-1} =0$ for all $n\geq k+1$. Hence, from Proposition \ref{prop1}, we can conclude that $a_n$ is a polynomial of degree at most $k$. 
\end{proof}

\item  {\bf   Closure properties:} 
\begin{thm}
Assume that $a_n$ and $b_n$ are polynomial sequences of degree $k$ and $l$, respectively.
The following sequences are also polynomial sequences.
\begin{enumerate}[label=(\roman*)]
\item addition: $(a_n+b_n)_{n=0}^{\infty}$, degree at most $\max(k,l)$,
\item term-wise multiplication: $(a_n \cdot b_n)_{n=0}^{\infty}$, degree at most $k+l$,
\item Cauchy product: $(\sum_{i=0}^n a_i \cdot b_{n-i})_{n=0}^{\infty}$, degree at most $k+l+1$,
\item partial sum: $(\sum_{i=0}^n a_i)_{n=0}^{\infty}$, degree at most $k+1$,
\item linear subsequence: $(a_{mn})_{n=0}^{\infty}$, where $m$ is a positive integer, degree at most $k$. 
\end{enumerate}
\end{thm}

\begin{proof}
The proofs of claims \emph{(i)},\emph{(ii)}, and \emph{(v)}
are rather straightforward. 
Claims \emph{(iii)} and \emph{(iv)} follow from Proposition \ref{prop2} that 
$ \displaystyle \sum_{i=0}^n i^k,$ where $k$ is a non-negative integer, 
is a polynomial in $n$ of degree at most $k+1$.
\end{proof}
\end{enumerate}

\subsection{$C$-finite}
\begin{enumerate}[leftmargin=0.4in,label=\textbf{\arabic*.}]
\item  {\bf   Ansatz:} $a_n$ is defined by a homogeneous
linear recurrence with constant coefficients:
\[ a_{n+r} + c_{r-1}a_{n+r-1}+ c_{r-2}a_{n+r-2}+ \dots +c_{0}a_{n} = 0,\]
along with the initial values $a_0, a_1, \dots, a_{r-1}.$
We call $a_n$ a \emph{$C$-finite sequence} of order $r$.
\begin{rem}
From Proposition \ref{prop1}, a polynomial sequence is a special case of $C$-finite sequences.
\end{rem}

\item  {\bf   Example:} Let $ \displaystyle a_n = \left \lfloor{\left(\dfrac{n}{2}\right)^2}\right \rfloor.$ \;\
Here, $a_n$ satisfies the linear recurrence relation of order 4: 
\[ a_{n+4}-2a_{n+3}+2a_{n+1}-a_n=0,\]
with initial values $a_0=0, a_1=0, a_2=1$ and $a_3=2.$
In terms of the left shift operator,
\[   0 = (N^4-2N^3+2N-1) \cdot a_n = (N+1)(N-1)^3 \cdot a_n.  \]

\item  {\bf   Guessing:} 

Input: the order $r$ of linear recurrence and 
sequence $a_n$ (of length more than $2r$).
We solve the system of linear equations for $c_0, c_1, \dots, c_{r-1}$

\[  \begin{bmatrix}
c_0 & c_1 & \dots & c_{r-1} & 1
\end{bmatrix} 
\begin{bmatrix}
a_0 & a_1 & a_2 & \dots & a_{r-1}\\
a_1 & a_2 & a_3 & \dots & a_r\\
& & \dots \\
a_{r-1} & a_{r} & a_{r+1}  & \dots & a_{2r-2}\\
a_r & a_{r+1} & a_{r+2} & \dots & a_{2r-1}
\end{bmatrix} =
\begin{bmatrix}
0 & 0 & \dots & 0
\end{bmatrix}.
\]

This matrix equation can be solved quickly through, say, the reduced row echelon form method.
Once all these $c_i$'s are obtained, we check that the rest of the $a_n,  n > 2r-1$ satisfy the conjectured recurrence.

\begin{code}
> A:= [seq(floor((n/2)^2), n=0..30)];
> GuessC(A,N);
\end{code}



\item  {\bf   Generating function:} 

Let us denote by $T(N)$ an \emph{annihilator} of $a_n$, that is, $T(N)\cdot a_n = 0$. In line with Proposition \ref{prop:generating_polynomial} in the previous section, the following proposition establishes a relationship between $C$-finite sequences and generating functions.

\begin{prop}
\label{prop:generating_cfinite}
Let $\displaystyle f(x) = \sum_{n=0}^{\infty} a_nx^n $
where $a_n$ is a $C$-finite sequence of order $r$ with annihilator
$T(N)$, a polynomial in $N$ of degree $r$. Then,
\begin{equation}
\label{eq:f_cfinite}
    f(x) = \dfrac{P(x)}{x^rT(1/x)},
\end{equation}
for some polynomial $P(x)$ of degree at most $r-1$.
\end{prop}

\begin{proof}
Assume that $a_n$ is a $C$-finite sequence with annihilator $T(N)$. 
Suppose further that $T(N) = c_rN^r+c_{r-1}N^{r-1}+\dots+c_1N+c_0,$ 
where $c_r=1$ and $c_{r-1},\dots, c_{1}, c_0$ are some constants. Then,
\begin{align*}
x^rT(1/x)f(x) &= \sum_{n=0}^{\infty} a_nx^{n}(c_0x^r+c_1x^{r-1}+\dots+c_r)\\
&=  \sum_{n=0}^{\infty}\sum_{i=0}^{r} c_i a_nx^{n+r-i} = \sum_{i=0}^{r} \sum_{n=r-i}^{\infty} c_i a_{n-r+i}x^{n} \\
& = \sum_{n=r}^{\infty} \sum_{i=0}^{r} c_i a_{n-r+i}x^{n}
 + \sum_{n=0}^{r-1} \sum_{i=r-n}^{r} c_i a_{n-r+i}x^{n}.
\end{align*}

The first sum is zero as, for each $n \geq r$, \\
$\displaystyle  \sum_{i=0}^{r} c_i a_{n-r+i} 
= \sum_{i=0}^{r} c_iN^ia_{n-r} = T(N) \cdot a_{n-r} = 0,$
by assumption of $T(N)$. Hence,
\[ x^rT(1/x)f(x) = P(x), \]
where $\displaystyle P(x) = \sum_{n=0}^{r-1} \sum_{i=r-n}^{r} c_i a_{n-r+i}x^{n}$,
a polynomial of degree at most $r-1$.
\end{proof}

{\bf Example:} The generating function $f(x)$ for $a_n 
= \left \lfloor{\left(\dfrac{n}{2}\right)^2}\right \rfloor$ is
$\dfrac{x^2}{(1+x)(1-x)^3}.$

\begin{code}
> GenC(N^4-2*N^3+2*N-1,[0,0,1,2],N,x);
\end{code}



The next proposition gives the converse of Proposition \ref{prop:generating_cfinite}. This proposition is very useful as it allows us to prove closure properties through generating functions which turned out to simplify our proof tremendously. In particular, the proposition implies that an upper bound for the order of a $C$-finite sequence can be determined by looking at the degree of the polynomial appearing in the denominator of \eqref{eq:f_cfinite}.

\begin{prop}
\label{prop:converse_cfinite}
Assume that $\displaystyle f(x) = \sum_{n=0}^{\infty} a_nx^n $ satisfies 
\[
    f(x) = \dfrac{P(x)}{x^rT(1/x)},
\]
for some polynomial $P(x)$ of degree $r-1$, and $T(N)$ is a polynomial in $N$ of degree $r$.
Then, $a_n$ is a $C$-finite sequence of order at most $r$.
\end{prop}

\begin{proof}
Since 
\[
\displaystyle  x^rT(1/x)f(x)= \sum_{n=r}^{\infty} \sum_{i=0}^{r} c_i a_{n-r+i}x^{n}
 + \sum_{n=0}^{r-1} \sum_{i=r-n}^{r} c_i a_{n-r+i}x^{n}
\]
is a polynomial of degree $r-1$, the first summation must be zero implying $T(N) \cdot a_{n-r} = 0$ for all $n\geq r$. Thus, $T(N)$ is an annihilator of $a_n$, and so $a_n$ is a $C$-finite sequence of order at most $r$. 
\end{proof}

\item  {\bf   Closure properties:} 

It will become evident later that knowing the operations under which the class of $C$-finite sequences is closed allows one to guess rigorously a formula of the resulting sequence. We first state and prove the following closure properties.

\begin{thm}     \label{Cclose}
Assume that $a_n$ and $b_n$ are $C$-finite sequences
of order $r$ and $s$, respectively.
The following sequences are also $C$-finite, with the specified upper bound on the order. 
\begin{enumerate}[label=(\roman*)]
\item addition: $(a_n+b_n)_{n=0}^{\infty}$,  order at most $r+s$,
\item term-wise multiplication: $(a_n \cdot b_n)_{n=0}^{\infty}$, order at most $rs$,
\item Cauchy product: $(\sum_{i=0}^n a_i \cdot b_{n-i})_{n=0}^{\infty}$,  order at most $r+s$,
\item partial sum: $(\sum_{i=0}^n a_i)_{n=0}^{\infty}$,  order at most $r+1$,
\item linear subsequence: $(a_{mn})_{n=0}^{\infty}$, where $m$ is a positive integer,
order at most $r$.
\end{enumerate}
\end{thm}

\begin{proof}
The proofs for the closure properties are based on two different approaches, i.e. the generating function approach for proving \emph{(i)}, \emph{(iii)} and \emph{(iv)}, and the solution subspace approach for \emph{(ii)} and \emph{(v)}. 

\vspace{1em}


\textbf{Generating function approach}

To prove the closure properties of addition, Cauchy product and partial sum, let $A(x)$ and $B(x)$ be the generating functions of $a_n$ and $b_n$, respectively. Then, the generating functions $C(x)$ of $\displaystyle c_n = a_n+b_n, \sum_{i=0}^n a_i \cdot b_{n-i}$
and $\displaystyle \sum_{i=0}^n a_i$ are $\displaystyle A(x)+B(x), A(x)B(x)$ and $\displaystyle A(x)\cdot\dfrac{1}{1-x}$, respectively. 
By Proposition \ref{prop:converse_cfinite}, $c_n$ is a $C$-finite sequence whose order can be determined by looking at the degree of the polynomial appearing in the denominator of $C(x)$ in each case. \\
\end{proof}

We now give concrete examples of how the generating function approach can be used to verify the closure properties of \emph{(i)}, \emph{(iii)} and \emph{(iv)}. \vspace{1em}

{\bf Example:} Let $a_n = \left \lfloor{\left(\dfrac{n}{2}\right)^2}\right \rfloor$
and $b_n$ be the Fibonacci numbers.

We recall that $a_n$ satisfies the linear recurrence relation
\[ a_{n+4}-2a_{n+3}+2a_{n+1}-a_n=0, \]
with  $a_0=0, a_1=0, a_2=1$ and $a_3=2.$
In terms of the shift operator,  \[ (N+1)(N-1)^3 \cdot a_n = 0. \]

Also, $b_n$ satisfies the linear recurrence relation 
\[ b_{n+2} -b_{n+1}-b_{n} = 0, \] with $b_0=0$ and $b_1=1.$  
In terms of the shift operator,  \[ (N^2-N-1) \cdot b_n = 0. \]

The generating function of $a_n$ and $b_n$ are
$A(x) =  \dfrac{x^2}{(1+x)(1-x)^3}$, and 
$B(x) = \dfrac{x}{1-x-x^2}$, respectively.

\begin{itemize}[leftmargin=0.2in]
\item addition: $c_n=a_n+b_n$. Then,  \[ C(x) = A(x)+B(x)
= \dfrac{-x(x^4-x^3+x^2+x-1)}{(1+x)(1-x)^3(1-x-x^2)}.\]
That is, $c_n$ satisfies the linear recurrence of order 6: 
\[
    (N+1)(N-1)^3(N^2-N-1) \cdot c_n = 0.
\]
\item Cauchy product: $ c_n = \sum_{i=0}^n a_i \cdot b_{n-i}$.
Then,  \[ C(x) = A(x)B(x)
= \dfrac{x^3}{(1+x)(1-x)^3(1-x-x^2)}. \]
That is, $c_n$ satisfies the linear recurrence of order 6: 
\[ 
(N+1)(N-1)^3(N^2-N-1) \cdot c_n = 0.
\]
\item partial sum: $c_n = \sum_{i=0}^n a_i$ then    
     \[ C(x) = A(x)\cdot \dfrac{1}{1-x}
= \dfrac{x^2}{(1+x)(1-x)^4}. \]
That is, $c_n$ satisfies the linear recurrence of order 5: 
\[ 
(N+1)(N-1)^4 \cdot c_n = 0.
 \]
\end{itemize}

\begin{code}
> A := x^2/(1+x)/(1-x)^3:
> B := x/(1-x-x^2):
> CAddition(A,B,x);
> CCauchy(A,B,x);
> CParSum(A,x);
\end{code}


\textbf{Solution space approach} 

Before embarking on a proof for closure properties of linear subsequence and term-wise multiplication, we continue with the example from the previous section, intended to highlight the key steps of the solution space approach. A formal proof will be deferred to the end of this section.

\begin{itemize}[leftmargin=0.2in]
\item linear subsequence: $c_n = a_{2n}$. 
First, we apply the linear recurrence relation of $a_{2n}$, i.e.
$a_{2n+4}=2a_{2n+3}-2a_{2n+1}+a_{2n}$ 
repeatedly to yield
\[  \begin{bmatrix}
c_n\\
c_{n+1}\\
c_{n+2}\\
c_{n+3} \\
c_{n+4}
\end{bmatrix} =   
 \begin{bmatrix}
1 & 0 & 0 & 0 \\
0 & 0 & 1 & 0 \\
1 & $-2$ & 0 & 2  \\
4 & $-6$ & $-3$ & 6  \\
9 & $-12$ & $-8$ & 12  \\
\end{bmatrix}  \cdot
 \begin{bmatrix}
a_{2n}\\
a_{2n+1}\\
a_{2n+2}\\
a_{2n+3}
\end{bmatrix}.\]

A solution $P=(-1,3,-3,1,0)$ is in the left null space of matrix $M$ (in the middle), i.e. $P \cdot M = (0,0,0,0)$.
Hence, the linear recurrence of $c_n$ is
\[ -c_n+3c_{n+1}-3c_{n+2}+c_{n+3} = 0  \]
or in terms of the shift operator
\[  (N-1)^3\cdot c_n = 0.  \]

\begin{code}
> CSubSeq(2,N^4-2*N^3+2*N-1,N);
\end{code}

Note that the above program returns $(N-1)^3(Nd_4-1)\cdot c_n = 0$, where $d_4$ is a free variable. We can assign zero to $d_4$ and obtain the desired third order recurrence relation.\\



\item term-wise multiplication: $c_n = a_n \cdot b_n$.
We apply the linear recurrence relation of $a_n$ and $b_n$, i.e.
$a_{n+4}=2a_{n+3}-2a_{n+1}+a_n$ and 
$b_{n+2} =b_{n+1}+b_{n}$ repeatedly 
to yield the system of relations
\begin{align*}
a_nb_n &= 1a_nb_n + 0a_nb_{n+1} + 0a_{n+1}b_{n} + 0a_{n+1}b_{n+1}
+ \dots + 0a_{n+3}b_{n+1}, \\
a_{n+1}b_{n+1} &= 0a_nb_n + 0a_nb_{n+1} + 0a_{n+1}b_{n} + 1a_{n+1}b_{n+1}
+ \dots + 0a_{n+3}b_{n+1}, \\
& \vdots  \\
a_{n+8}b_{n+8} &= 117a_nb_n + 189a_nb_{n+1} -156a_{n+1}b_{n} -252a_{n+1}b_{n+1}
+ \dots + 252a_{n+3}b_{n+1}.
\end{align*}

We put this system of equations in the matrix form as
\[  \begin{bmatrix}
c_n\\
c_{n+1}\\
c_{n+2}\\
c_{n+3}\\
c_{n+4}\\
c_{n+5}\\
c_{n+6}\\
c_{n+7} \\
c_{n+8}
\end{bmatrix} =   
 \begin{bmatrix}
1 & 0 & 0 & 0 & 0 & 0 & 0 & 0\\
0 & 0 & 0 & 1 & 0 & 0 & 0 & 0\\
0 & 0 & 0 & 0 & 1 & 1 & 0 & 0\\
0 & 0 & 0 & 0 & 0 & 0 & 1 & 2\\
2 & 3 & -4 & -6 & 0 & 0 & 4 & 6\\
6 & 10 & -9 & -15 & -6 & -10 & 12 & 20\\
20 & 32 & -30 & -48 & -15 & -24 & 30 & 48\\
48 & 78 & -64 & -104 & -48 & -78 & 72 & 117\\
117 & 189 & -156 & -252 & -104 & -168 & 156 & 252
\end{bmatrix}  \cdot
 \begin{bmatrix}
a_nb_n\\
a_{n}b_{n+1}\\
a_{n+1}b_{n}\\
a_{n+1}b_{n+1}\\
a_{n+2}b_{n}\\
a_{n+2}b_{n+1}\\
a_{n+3}b_{n}\\
a_{n+3}b_{n+1}
\end{bmatrix}.\]

A non-trivial solution  $(1,2,-4,-8,5,8,-4,-2,1)$ is in the left null space of matrix $M$ (in the middle). 
Hence, the linear recurrence of $c_n$ is
\[ c_n+2c_{n+1}-4c_{n+2}-8c_{n+3}+5c_{n+4}+8c_{n+5}-4c_{n+6}-2c_{n+7}+c_{n+8} = 0  \]
or in terms of the shift operator
\[  (N^2+N-1)(N^2-N-1)^3\cdot c_n = 0.  \]
\end{itemize}

\begin{code}
> CTermWise(N^4-2*N^3+2*N-1,N^2-N-1,N);
\end{code}



\vspace{1em}
We shall now give a formal proof to the closure properties of linear subsequence and term-wise multiplication in the spirit of the last two examples.

\begin{proof}
For the case of the linear subsequence with a fixed positive integer $m$, 
let $c_n=a_{mn}$. Then $c_{n}, c_{n+1},\dots, c_{n+r}$
can be put in the system of linear equations as
\[  \begin{bmatrix}
c_n\\
c_{n+1}\\
c_{n+2}\\
\dots \\
c_{n+r} 
\end{bmatrix} =   
 \begin{bmatrix}
1 & 0 & 0 & \dots &  \\
&&&\dots \\
M_{r-1}^{(1)} & M_{r-1}^{(2)} & M_{r-1}^{(3)} &   \dots &  &  &  &  \\
M_{r}^{(1)} & M_{r}^{(2)} & M_{r}^{(3)} &   \dots &  &  &  &  
\end{bmatrix}  \cdot
 \begin{bmatrix}
a_{mn}\\
a_{mn+1}\\
a_{mn+2}\\
\dots\\
a_{mn+r-1}
\end{bmatrix}.\]
The constant matrix $M$ (in the middle) has $r+1$ rows
and $r$ columns, which guarantees a non-trivial null space, i.e. there exists a solution $P\neq 0$ such that $P \cdot M = [0\, 0 \cdots 0]$.
This solution $P$ provides a $C$-finite recurrence relation
to $c_n, c_{n+1},\dots, c_{n+r}.$
 
As for the term-wise multiplication, 
 let $c_n=a_n\cdot b_n$. Then, $c_n, c_{n+1},\dots, c_{n+rs}$
can be put in the system of linear equations as
\[  \begin{bmatrix}
c_n\\
c_{n+1}\\
c_{n+2}\\
\dots \\
c_{n+rs-1} \\
c_{n+rs}
\end{bmatrix} =   
 \begin{bmatrix}
1 & 0 & 0 & \dots & 0 & 0 & 0 & \dots \\
0 & 0 & 0 & \dots & 0 & 1 & 0 & \dots \\
&&&\dots \\
M_{rs-1}^{(1)} & M_{rs-1}^{(2)}& M_{rs-1}^{(3)} &   \dots &  &  &  &  \\
M_{rs}^{(1)} & M_{rs}^{(2)}& M_{rs}^{(3)} &   \dots &  &  &  &  
\end{bmatrix}  \cdot
 \begin{bmatrix}
a_nb_n\\
a_{n}b_{n+1}\\
\dots\\
a_{n+1}b_n\\
a_{n+1}b_{n+1}\\
\dots\\
a_{n+r-1}b_{n+s-1}
\end{bmatrix}.\]
The constant matrix $M$ (in the middle) has $rs+1$ rows
and $rs$ columns, which again guarantees a 
non-trivial solution $P$ in the null space of $M$. 
This solution $P$ gives a $C$-finite recurrence relation to $c_n,c_{n+1},\dots, c_{n+rs}.$
\end{proof}

\textbf{Rigorous proof of identities with the closure properties}

The closure properties from Theorem \ref{Cclose} are extremely useful in verifying identities of $C$-finite sequences. 
For example, let $c_n = a_n^2$, where 
$a_n = \left \lfloor{\left(\dfrac{n}{2}\right)^2}\right \rfloor$.
Let us say we want to show that $c_n$ satisfies a linear recurrence
\[  c_{n+8} = 2c_{n+7}+2c_{n+6}-6c_{n+5}+6c_{n+3}-2c_{n+2}-2c_{n+1}+c_{n},   \;\ n \geq 0, \]
i.e. $(N+1)^3(N-1)^5 \cdot c_n = 0.$
Knowing that $a_n$ satisfies the linear recurrence of order 4, Theorem \ref{Cclose}: term-wise multiplication guarantees
that $c_n$ satisfies the linear recurrence of order at most $4\cdot4 = 16.$
It then suffices to verify the identity only by checking the numeric values for $n=0,1,2,\dots,15$
as this is adequate to determine a $C$-finite recurrence of order 16.

Let us give another example. Consider the same sequence $a_n$, but this time we want to verify a non-linear identity
\[  a_{n+1}-a_na_{n+1}+a_na_{n+2}+a_{n+1}^2-a_{n+1}a_{n+2}=0,  \;\ n \geq 0. \]
We define $d_n = a_{n+1}-a_na_{n+1}+a_na_{n+2}+a_{n+1}^2-a_{n+1}a_{n+2}$ for $n \geq 0.$
The closure properties of $C$-finite sequence ensure that the order of $d_n$ will be at most
$4+4^2+4^2+4^2+4^2=68.$ So to prove that $d_n=0,$ for $n \geq 0$,
we only check the initial values of $d_n$ for $n=0,1,2,\dots,67$.

\item  {\bf   Closed-form solutions:} 

In contrast to polynomial sequences that we started with an ansatz in a closed-form solution (and later derived a linear recurrence relation for it), our ansatz for a $C$-finite sequence was initially defined as a recurrence relation. A closed-form representation will now be given in the following proposition. 

\begin{prop}
\label{prop:close_cfinite}
Assume that $a_n$ satisfies the linear recurrence with constant coefficients,
\[  0 = T(N) \cdot a_n = [(N-\alpha_1)^{k_1} (N-\alpha_2)^{k_2}
\dots (N-\alpha_m)^{k_m}]\cdot a_n,\]
where $\alpha_i $, the $i$-th root of the characteristic polynomial of the recurrence, has multiplicity $k_i,\, i = 1, 2,\dots, m$.  
Then, a closed-form formula of $a_n$ is given by
\begin{equation}
\label{eq:closed_form_cfinite}
    a_n = \sum_{i=1}^m(c_{i,0}+c_{i,1}n+\dots+c_{i,k_i-1}n^{k_i-1})\alpha_i^n, 
\end{equation}   
for some constants $c_{i,0}, c_{i,1}, \dots, c_{i,k_i-1}, 
\;\ i=1,2,\dots,m,$ which can be determined by the initial conditions.
\end{prop}

\begin{proof}
We will show that $\displaystyle T(N) \cdot \sum_{i=1}^m(c_{i,0}+c_{i,1}n+\dots+c_{i,k_i-1}n^{k_i-1})\alpha_i^n = 0.$
This can be done by showing that each part of the sum is 0, i.e. for each $i$,
\[ (N-\alpha_i)^{k_i} \cdot [(c_{i,0}+c_{i,1}n+\dots+c_{i,k_i-1}n^{k_i-1})\alpha_i^n] = 0.\]
Notice that (dropping the $i$ subscript for simplicity)
\[ (N-\alpha)\cdot [P(n)\alpha^n] = P(n+1)\alpha^{n+1}-  P(n)\alpha^{n+1} = Q(n)\alpha^{n+1}, \]
where $Q(n)$ has degree less than $P(n)$. 
Therefore, if we apply $(N-\alpha)$ for $k$ times to $P(n)\alpha^n$
where $P(n)$ is a polynomial of degree $k-1$, we will get 0, and the result is now immediate. 
We note that this formula of $a_n$ is the most general form 
as the number of independent solutions equals the degree of $T(N)$.
\end{proof}

{\bf Example:} A closed-form formula for 
$a_n = \left \lfloor{\left(\dfrac{n}{2}\right)^2}\right \rfloor$ is
$\dfrac{n^2}{4}-\dfrac{1}{8}+\dfrac{(-1)^n}{8}.$

\begin{code}
> CloseC(N^4-2*N^3+2*N-1,[0,0,1,2],0,N,n);
\end{code}



The next proposition, which is the converse of the previous proposition, ensures that any sequence expressed in the form \eqref{eq:closed_form_cfinite} is also a $C$-finite sequence. This proposition will be used in a later section to prove some results for the $C^2$-finite sequences.

\begin{prop}
\label{prop:converse_closeform_cfinite}
Suppose that 
\[  a_n = \sum_{i=1}^m(c_{i,0}+c_{i,1}n+\dots+c_{i,k_i-1}n^{k_i-1})\alpha_i^n,  \]
for some constants $c_{i,0}, c_{i,1}, \dots, c_{i,k_i-1}, \text{and } \alpha_i, 
\;\ i=1,2,\dots,m,$.
Then, $a_n$ satisfies a linear recurrence with constant coefficients, i.e $a_n$ is a $C$-finite sequence. Moreover, the order of $a_n$ is $k_1+k_2+\dots+k_m$.
\end{prop}

\begin{proof}
Adopt the same steps followed in the proof of Proposition \ref{prop:close_cfinite} to get 
\[ [(N-\alpha_1)^{k_1} (N-\alpha_2)^{k_2}
\dots (N-\alpha_m)^{k_m}]\cdot a_n=0.\]
This relation together with initial values $a_n$ of length $k_1+k_2+\dots +k_m$ defines a $C$-finite recurrence relation for $a_n$.
\end{proof}
\end{enumerate}
\subsection{Holonomic}

\begin{enumerate}[leftmargin=0.4in,label=\textbf{\arabic*.}]
\item  {\bf   Ansatz:} $a_n$ is defined by a linear recurrence with polynomial coefficients:
\[ p_r(n)a_{n+r} + p_{r-1}(n)a_{n+r-1}+ \dots +p_{0}(n)a_{n} = 0, \]
where  $p_r(n) \neq 0,  n=0,1,2, \dots$,
along with the initial values $a_0, a_1, \dots, a_{r-1}.$
We call $a_n$ a \emph{holonomic sequence} of order $r$
and degree $k$, where $k$ is the highest degree amongst the
polynomial $p_r(n), p_{r-1}(n), \dots, p_0(n).$

It is important to note that the condition $p_r(n) \neq 0$ for the leading coefficient is necessary for recursively computing the term $a_{n+r}$ in the sequence from its predecessors. In case of violation of the condition, the relation will be valid for $a_{n}, n>n_0$ where $n_0$ is the largest positive integer root of the equation $p_r(n)=0$.

\item  {\bf   Example 1:} Let $ \displaystyle a_n =\dfrac{1}{n+1}\binom{2n}{n}$, the Catalan numbers. \;\
Here, $a_n$ satisfies a holonomic recurrence of order 1 and degree 1: 
\[ (4n+2)a_n-(n+2)a_{n+1}=0,\]
with initial values $a_0=1.$
In terms of the shift operator,
\[  [(4n+2)-(n+2)N] \cdot a_n = 0.  \]

{\bf Example 2:} Let $ \displaystyle a_n =\sum_{i=1}^n \dfrac{1}{i}$, the harmonic numbers. \;\
Here, $a_n$ satisfies a holonomic recurrence of order 2 and degree 1: 
\[ (n+1)a_n-(2n+3)a_{n+1}+(n+2)a_{n+2} =0,\]
with initial values $a_1=1, a_2=\frac{3}{2}.$
In terms of the shift operator,
\[  [n+1-(2n+3)N+(n+2)N^2] \cdot a_n = 0.  \]

{\bf Example 3:} Let $ \displaystyle a_n = \left \lfloor{\left(\dfrac{n}{2}\right)^2}\right \rfloor,$
from the previous section. 
Here, $a_n$ also satisfies a holonomic recurrence of order 2 and degree 1: 
\[ (n+2)a_n+2a_{n+1}-na_{n+2} =0.\]
This example illustrates the trade-off between lower order and higher degree compared to the $C$-finite recurrence in the previous section.

\item  {\bf   Guessing:} 

Input: the order $r$ and degree $k$ of linear recurrence and 
sequence $a_n$ (of length more than $(r+1)(k+1)-1+r$).
Let us write $p_i(n) = \sum_{j=0}^k c_{i,j}n^j$,
$C_i = \begin{bmatrix}
c_{i,0}, c_{i,1}, \dots, c_{i,k}
\end{bmatrix}$ and 
$P_n = \begin{bmatrix}
1, n, n^2, \dots, n^k
\end{bmatrix}.$ 

In  matrix notation, the system of linear equations for $c_{i,j},  0 \leq i \leq r, 0 \leq j \leq k$ takes the following form:
\[  \begin{bmatrix}
C_0, C_1, \dots, C_r
\end{bmatrix} \cdot 
\begin{bmatrix}
P_0^Ta_0 & P_1^Ta_1 & \dots & P_{(r+1)(k+1)-1}^Ta_{(r+1)(k+1)-1}\\
P_0^Ta_1 & P_1^Ta_2 & \dots & P_{(r+1)(k+1)-1}^Ta_{(r+1)(k+1)}\\
& & \dots \\
P_0^Ta_{r} & P_1^Ta_{r+1}  & \dots & P_{(r+1)(k+1)-1}^Ta_{(r+1)(k+1)-1+r}\\
\end{bmatrix} =
\begin{bmatrix}
0 & 0 & \dots & 0
\end{bmatrix}.
\]

A non-trivial solution $c_{i,j}$'s can be found quickly
using e.g. the reduced row echelon form method.
Once we get all these $c_{i,j}$'s, we check with
the rest of the $a_n,  n > (r+1)(k+1)-1+r$, that they satisfy the conjectured recurrence.

\begin{code}
> A:= [seq(add(1/i,i=1..n),n=1..35)];
> GuessHo(A,2,1,1,n,N);
\end{code}

{\bf Exercise:} The sequence 1, 1, 5, 23, 135, 925, 7285, 64755, 641075, 6993545, 83339745,..., appearing in the abstract, is not contained in the OEIS database. Convince yourself that this sequence has a recurrence relation of the form \[a_{n+2} = (n+3)a_{n+1}+(n+2)a_{n} \] with $a_0=a_1=1$. \vspace{1em}

\textbf{Guessing turns into Rigorous Proof}

We note that guessing can turn into a rigorous proof,
if we happen to know the order and degree of 
the relation of the holonomic sequence, as illustrated in the following example. The subject of finding bounds of the order and degree will be discussed throughout this section.

{\bf Example:} Given that we know $a_n$ is a 
holonomic sequence of order 2 and degree 3, then $a_n$ must satisfy the recurrence relation
\begin{align*}
& (c_{2,0}+c_{2,1}n+c_{2,2}n^2+c_{2,3}n^3)\cdot a_{n+2} 
+ (c_{1,0}+c_{1,1}n+c_{1,2}n^2+c_{1,3}n^3)\cdot a_{n+1} \\
& + (c_{0,0}+c_{0,1}n+c_{0,2}n^2+c_{0,3}n^3)\cdot a_{n}   =0, 
\end{align*}
for some unknown $c_{i,j}, 0 \leq i \leq 2, 0 \leq j \leq 3$. 
To find the holonomic relation, we simply fit this equation to some data $a_n$, where we need at least  $a_n$, $n=0,1,\dots,13$ to solve for the 12 unknowns.

\item  {\bf   Generating function:} 
\begin{thm}  \label{GenHolo}
Let $\displaystyle f(x) = \sum_{n=0}^{\infty} a_nx^n $
where $a_n$ is a holonomic sequence of order $r$ and degree $k$.
Then, $f(x)$ satisfies the (non-homogeneous) 
linear differential equation with polynomial coefficients
\begin{equation}    \label{diff}
q_0(x)f(x)+ q_1(x)f'(x)+ \dots + q_{r'}(x)f^{(r')}(x) = R(x),  
\end{equation}
where order $r'$ is at most $k$, degree of the coefficient $q_t(x)$ for each $t$
is at most $r+k$, and degree of the polynomial $R(x)$ is at most $r-1.$
\end{thm}

\begin{defi}
The generating function $f(x)$ of a holonomic sequence is called \emph{holonomic}.
Also, $f(x)$ satisfying \eqref{diff} is called \emph{$D$-finite} (differentiably finite),
see \cite{S}.
\end{defi}

\begin{proof}
Assume that $a_n$ satisfies the relation   
\[ 
p_r(n)a_{n+r} + p_{r-1}(n)a_{n+r-1}+ \dots +p_{0}(n)a_{n} = 0.
\]

We denote by $b_{s,t}$ the coefficient of $a_{n+t}n^s$ in the relation above, that is,
$\displaystyle p_t(n)=\sum_{s=0}^kb_{s,t}n^s$, and so the holonomic relation becomes 
\begin{equation} 
\label{holo8}
\sum_{t=0}^r\sum_{s=0}^k b_{s,t}n^sa_{n+t}=0.
\end{equation}

To prove our results, first we note the following identity. For fixed $s$ and $t$, 
\begin{align*}
\sum_{j=0}^s c^{(s,t)}_jx^{j-t}f^{(j)}(x) 
&=  \sum_{j=0}^s c^{(s,t)}_j  \sum_{n=0}^{\infty} (n)_ja_nx^{n-t} \\
&=  \sum_{j=0}^s c^{(s,t)}_j  \sum_{n=-t}^{\infty} (n+t)_ja_{n+t}x^{n} \\
&=  \sum_{n=-t}^{\infty} \left[ \sum_{j=0}^s c^{(s,t)}_j (n+t)_j \right] a_{n+t}x^{n},
\end{align*}
where $(n)_j$ is a falling factor, i.e. $(n)_j = n(n-1)\dots(n-j+1)$, and $c^{(s,t)}_j$'s are some constants.

For each pair of $(s,t)$, we appeal to the method of equating the coefficients to obtain $c^{(s,t)}_j, \;\ j=0,1,2,\dots,s$.
Equating the corresponding coefficients of $n^j$ in the equation $\displaystyle \sum_{j=0}^s c^{(s,t)}_j (n+t)_j = n^s$ results in the system of $s+1$  linear equations with $s+1$ unknowns, and so the unknown constants  $c^{(s,t)}_j$ can be determined.  

Next, we define $\displaystyle A_{s,t}(x) = \sum_{n=0}^{\infty} a_{n+t}n^sx^n$ for $s,t \geq 0.$ Then,
\begin{equation} \label{tank} 
\sum_{j=0}^s c^{(s,t)}_jx^{j-t}f^{(j)}(x) 
=  \sum_{n=-t}^{\infty} a_{n+t}n^sx^{n}
= A_{s,t}(x) + \sum_{n=-t}^{-1} a_{n+t}n^sx^{n}.
\end{equation}

From the holonomic relation of $a_n$ in \eqref{holo8}, 
multiply $x^n$ through and sum $n$ from $0$ to $\infty$:
\begin{align*} 
0 &= \sum_{n=0}^\infty\sum_{t=0}^r\sum_{s=0}^k b_{s,t}n^sa_{n+t}x^n= \sum_{t=0}^r\sum_{s=0}^k b_{s,t}A_{s,t}(x) \\
&= \sum_{t=0}^r\sum_{s=0}^k b_{s,t} 
\left[\sum_{j=0}^s c^{(s,t)}_jx^{j-t}f^{(j)}(x)
- \sum_{n=-t}^{-1} a_{n+t}n^sx^{n} \right] , \;\  \text{ from \eqref{tank}.}
\end{align*}
Multiplying $x^r$ on both sides and rearranging this equation, we obtain
\[    \sum_{t=0}^r\sum_{s=0}^k b_{s,t} 
\sum_{j=0}^s c^{(s,t)}_jx^{j+r-t}f^{(j)}(x)
= \sum_{t=0}^r\sum_{s=0}^k b_{s,t} 
\sum_{n=0}^{t-1} a_{n}(n-t)^sx^{n+r-t} . \]
Observe that the left hand side is the differential equation
of order at most $k$ and degree at most $r+k$, while the right hand side is the polynomial of degree at most $r-1$ as desired.
\end{proof}

{\bf Example 1:} Let $a_n=n!$. Then, $a_n$ satisfies the holonomic recurrence
\[  a_{n+1}-(n+1)a_n =0, \;\ a_0=1. \]
The differential equation corresponding to its generating function is
\[  (1-x)f(x) -x^2f'(x) = 1.  \]

\begin{code}
> HoToDiff(n+1-N,[1],n,N,x,D);
\end{code}



{\bf Example 2:} Let $ \displaystyle a_n =\dfrac{1}{n+1}\binom{2n}{n}$,
which satisfies the holonomic recurrence
\[ (4n+2)a_n-(n+2)a_{n+1}=0.\]
The differential equation corresponding to its generating function is
\[  (1-2x)f(x) +(x-4x^2)f'(x) = 1.  \]

It is well-known that a closed-form generating function of $f(x)$ is  $\dfrac{1-\sqrt{1-4x}}{2x}.$ 
The reader can easily verify that this expression of $f(x)$ satisfies the above differential equation. \\

\begin{code}
> HoToDiff(4*n+2-(n+2)*N,[1],n,N,x,D);
\end{code}



{\bf Example 3:} Let $ \displaystyle a_n = \left \lfloor{\left(\dfrac{n}{2}\right)^2}\right \rfloor$
with a holonomic recurrence given by
\[ (n+2)a_n+2a_{n+1}-na_{n+2} =0.\]
The differential equation corresponding to its generating function is
\[  2(1+x+x^2)f(x) +(-x+x^3)f'(x) = 0.  \]
On the other hand, another holonomic ($C$-finite) recurrence of $a_n$ is
\[ a_{n+4}-2a_{n+3}+2a_{n+1}-a_n=0,\]
which leads to the (zero-order) differential equation relation
\[  (1+x)(1-x)^3f(x) = x^2.  \]
The reader is encouraged to check that $f(x)$ from the zero-order relation satisfies the first-order differential equation. 

\begin{code}
> HoToDiff(n+2+2*N-N^2*n,[0,0],n,N,x,D);
> HoToDiff(1-2*N+2*N^3-N^4,[0,0,1,2],n,N,x,D);
\end{code}



A nonhomogeneous differential equation in \eqref{diff} 
can be transformed into a homogeneous one by differentiating multiple 
times until the polynomial $R(x)$ on the right hand side becomes zero. We state this fact in the following corollary.

\begin{cor}
\label{cor:homogeneous_holonomic}
Let $\displaystyle f(x) = \sum_{n=0}^{\infty} a_nx^n $
where $a_n$ is a holonomic sequence of order $r$ and degree $k$.
Then, $f(x)$ satisfies a \textit{homogeneous}
linear differential equation with polynomial coefficients,
\begin{equation}   
\label{eq:homoDiff_holo}
q_0(x)f(x)+ q_1(x)f'(x)+ \dots + q_{r'}(x)f^{(r')}(x) = 0,  
\end{equation}
where order $r'$ is at most $r+k$, and degree of $q_t(x)$ for each $t$
is at most $r+k$.
\end{cor}

\begin{rem}
Theorem 7.1.2 of \cite{S2} appears to contain a typographical error in the bound of the order $r'$ of the homogeneous linear differential equation \eqref{eq:homoDiff_holo}, having $k$ specified as a bound therein as opposed to $r+k$.
\end{rem}

\textbf{Example:} The homogeneous differential equation of $\displaystyle f(x) = \sum_{n=0}^{\infty} n!x^n$ is
\[x^2f''(x)+(3x-1)f'(x)+f(x)=0.\]

\begin{code}
> HoToDiffHom(n+1-N,[1],n,N,x,D);
\end{code}

The next theorem, which is the converse of Corollary \ref{cor:homogeneous_holonomic}, ensures that one can always establish a holonomic recurrence relation for the coefficients $a_n$ of $f(x)$ satisfying a homogeneous linear differential equation with polynomial coefficients.

\begin{thm}
\label{thm:generating_holonomic}
Let $\displaystyle f(x) = \sum_{n=0}^{\infty} a_nx^n. $
Assume $f(x)$ satisfies a homogeneous 
linear differential equation with polynomial coefficients
of order $r$ and degree $k$,
\begin{equation}    \label{diff2}
q_0(x)f(x)+ q_1(x)f'(x)+ \dots + q_{r}(x)f^{(r)}(x) = 0.  
\end{equation}
Then, $a_n$ is a holonomic sequence of order at most $r+k$ and degree 
at most $r$.
\end{thm}

\begin{proof}
For $0 \leq t \leq r$, let $b_{s,t}$ be the coefficient of $x^sf^{(t)}(x)$ so that $\displaystyle  q_t(x) = \sum_{s=0}^kb_{s,t}x^s.$
We note that, for each $s$ and $t$,
\[ x^sf^{(t)}(x) = \sum_{n=t}^{\infty} (n)_ta_nx^{n-t+s}
= \sum_{n=s}^{\infty} (n+t-s)_ta_{n+t-s}x^{n} . \]

Now \eqref{diff2} can be written as
\[ \sum_{t=0}^r \sum_{s=0}^{k} 
\sum_{n=s}^{\infty}  b_{s,t}(n+t-s)_ta_{n+t-s}x^{n} = 0 .   \]
Then, for each $\, n \geq 0, \, a_n$ satisfies the following recurrence
\[ \sum_{t=0}^r \sum_{s=0}^{\min(n,k)} b_{s,t}(n+t-s)_ta_{n+t-s} = 0,\]
and the claim about the bounds on the order and degree follows immediately.
\end{proof}

\item  {\bf   Closure properties:} 

\begin{thm}  \label{hoclose}
Assume that $a_n$ and $b_n$ are holonomic sequences
of order $r$ and $s$, respectively.
The following sequences are also holonomic sequences, with the specified upper bound on the order.
\begin{enumerate}[label=(\roman*)]
\item addition: $(a_n+b_n)_{n=0}^{\infty},$ of order at most $r+s,$
\item term-wise multiplication: $(a_n \cdot b_n)_{n=0}^{\infty},$ of order at most $rs,$
\item partial sum: $(\sum_{i=0}^n a_i)_{n=0}^{\infty}$,  of order at most $r+1,$
\item linear subsequence: $(a_{mn})_{n=0}^{\infty}$, where $m$ is a positive integer,
of order at most $r.$
\end{enumerate}

Furthermore, let $c_n=\sum_{i=0}^na_i\cdot b_{n-i}$ be the Cauchy product of $a_n$ and $b_n$. 
Assume that the generating functions of $a_n$ and $b_n$, denoted by $A(x)$ and $B(x)$, satisfy homogeneous differential equations of orders $r_1$ and $r_2$, respectively. Then, the generating function of $c_n$ also satisfies a homogeneous differential equation of order at most $r_1r_2$. 
\end{thm}

Let us reiterate here that if the resulting sequence under these operations has a zero leading coefficient $p_r(n)$ for some $n$, then the recurrence relation is only valid for $n>n_0$ where $n_0$ is the largest positive integer root of the equation $p_r(n)=0$.

\begin{proof}
To verify the closure properties \emph{(i)}-\emph{(iv)}, we follow the solution space approach 
in the same vein as that for the $C$-finite. As for the Cauchy product, it is worth pointing out that while the proof also relies on the solution space approach, this time we work with the generating function instead of the sequence itself.

We first give the proof to \emph{(i)}-\emph{(iv)}. Assume  $a_n$ is a holonomic sequence of order $r$, i.e.
\[ p_r(n)a_{n+r} + p_{r-1}(n)a_{n+r-1}+ \dots +p_{0}(n)a_{n} = 0, \]
and $b_n$ is a holonomic sequence of order $s$, i.e.
\[ q_s(n)b_{n+s} + q_{s-1}(n)b_{n+s-1}+ \dots +q_{0}(n)b_{n} = 0.\]
We note that for each fixed $k, \, k \geq r,$
$a_{n+k}$  can be written as a linear combination, 
with rational function coefficients, of  $a_{n+r-1}, a_{n+r-2}, \dots, a_{n}.$
This can be seen by repeated substitution starting from the term $a_{n+r}$, $a_{n+r+1}, \dots,$ all the way to $a_{n+k}$.
The same argument can be made for $b_{n+k}, k \geq s.$
\begin{itemize}[leftmargin=0.2in]
\item addition: let $c_n=a_n+b_n$. Then $c_n, c_{n+1},\dots, c_{n+r+s}$
can be put in the system of linear equations as
\[  \begin{bmatrix}
c_n\\
c_{n+1}\\
c_{n+2}\\
\dots \\
c_{n+r+s-1} \\
c_{n+r+s}
\end{bmatrix} =   
 \begin{bmatrix}
1 & 0 & 0 & \dots & 1 & 0 & 0 & \dots \\
0 & 1 & 0 & \dots & 0 & 1 & 0 & \dots \\
&&&\dots \\
M_{r+s-1}^{(1)}(n) & M_{r+s-1}^{(2)}(n) & M_{r+s-1}^{(3)}(n)&   \dots &  &  &  &  \\
M_{r+s}^{(1)}(n) & M_{r+s}^{(2)}(n) & M_{r+s}^{(3)}(n)&   \dots &  &  &  &  
\end{bmatrix}  \cdot
 \begin{bmatrix}
a_n\\
a_{n+1}\\
\dots\\
a_{n+r-1}\\
b_{n}\\
b_{n+1}\\
\dots\\
b_{n+s-1}
\end{bmatrix}.\]

Since the matrix $M$ with rational entries (rational matrix) in the middle 
has $r+s+1$ rows and $r+s$ columns, its null space is non-trivial, i.e. there exists a row vector $P\neq0$ such that $P\cdot M =[0,0,\dots,0].$ 
Note, in addition, that the solution $P$ in the form of rational functions has one free variable (as the number of rows $>$ number of columns). We can turn the solutions in $P$ to polynomials. This $P$ gives a holonomic relation to $c_n, c_{n+1}, \dots, c_{n+r+s}.$

\item term-wise multiplication: In a similar way as in the addition case, let $c_n=a_n\cdot b_n$. Then $c_n, c_{n+1},\dots, c_{n+rs}$
can be put in the system of linear equations as
\[  \begin{bmatrix}
c_n\\
c_{n+1}\\
c_{n+2}\\
\dots \\
c_{n+rs-1} \\
c_{n+rs}
\end{bmatrix} =   
 \begin{bmatrix}
1 & 0 & 0 & \dots & 0 & 0 & 0 & \dots \\
0 & 0 & 0 & \dots & 0 & 1 & 0 & \dots \\
&&&\dots \\
M_{rs-1}^{(1)}(n) & M_{rs-1}^{(2)}(n) & M_{rs-1}^{(3)}(n) &   \dots &  &  &  &  \\
M_{rs}^{(1)}(n) & M_{rs}^{(2)}(n) & M_{rs}^{(3)}(n) &   \dots &  &  &  &  
\end{bmatrix}  \cdot
 \begin{bmatrix}
a_nb_n\\
a_{n}b_{n+1}\\
\dots\\
a_{n+1}b_n\\
a_{n+1}b_{n+1}\\
\dots\\
a_{n+r-1}b_{n+s-1}
\end{bmatrix}.\]
Again since the matrix $M$ has $rs+1$ rows
and $rs$ columns, the null space of $M$ is non-trivial. A polynomial solution $P\neq0$ gives a holonomic relation to $c_n, c_{n+1},\dots, c_{n+rs}.$

\item linear subsequence: For a fixed integer $m$, let $c_n=a_{mn}$. Then, $c_{n}, c_{n+1},\dots, c_{n+r}$
can be put in the system of linear equations as
\[  \begin{bmatrix}
c_n\\
c_{n+1}\\
c_{n+2}\\
\dots \\
c_{n+r} 
\end{bmatrix} =   
 \begin{bmatrix}
1 & 0 & 0 & \dots & 0 \\
&&&\dots \\
M_{r-1}^{(1)}(mn) & M_{r-1}^{(2)}(mn) & M_{r-1}^{(3)}(mn) &   \dots & M_{r-1}^{(r)}(mn) \\
M_{r}^{(1)}(mn) & M_{r}^{(2)}(mn) & M_{r}^{(3)}(mn) &   \dots & M_{r}^{(r)}(mn)  
\end{bmatrix}  \cdot
 \begin{bmatrix}
a_{mn}\\
a_{mn+1}\\
a_{mn+2}\\
\dots\\
a_{mn+r-1}
\end{bmatrix}.\]
Here, the matrix $M$ has $r+1$ rows
and $r$ columns, so a non-trivial polynomial solution $P$
(in a null space of $M$) exists. This $P$ gives a holonomic relation to $c_n, c_{n+1},\dots, c_{n+r}.$

\item partial sum: We set up a slightly different matrix equation this time to simplify the computation. Let $c_n=\sum_{i=0}^na_i$. Then, $c_{n-1}, c_n, c_{n+1},\dots, c_{n+r}$ satisfy the following system of linear equations
\[  \begin{bmatrix}
c_n-c_{n-1}\\
c_{n+1}-c_{n}\\
c_{n+2}-c_{n+1}\\
\dots \\
c_{n+r}-c_{n+r-1}
\end{bmatrix} =   
 \begin{bmatrix}
1 & 0 & 0 & \dots & 0  \\
0 & 1 & 0 & \dots & 0  \\
&&&\ddots \\
0 & 0 & 0 & \dots & 1  \\   
-\dfrac{p_0(n)}{p_r(n)} & -\dfrac{p_1(n)}{p_r(n)} 
& -\dfrac{p_2(n)}{p_r(n)} & \dots & -\dfrac{p_{r-1}(n)}{p_r(n)}  
\end{bmatrix}  \cdot
 \begin{bmatrix}
a_{n}\\
a_{n+1}\\
\dots\\
a_{n+r-2}\\
a_{n+r-1}
\end{bmatrix}.\]
The matrix $M$ has $r+1$ rows and $r$ columns, so a non-trivial polynomial solution $P$ exists in the null space of $M$. This $P$ gives a holonomic relation of order $r+1$
to $c_{n-1}, c_n, c_{n+1},\dots, c_{n+r}.$\\

We now turn to the proof for the Cauchy product.
\item Cauchy product: 
Let $c_n=\sum_{i=0}^na_i\cdot b_{n-i}$. 
This time, we consider finding the homogeneous 
differential equation of $C(x)=\sum_{n=0}^{\infty}c_nx^n$.
Then we can invoke Theorem \ref{thm:generating_holonomic}, the relationship between $C(x)$ and the $c_n$ itself, to conclude that $c_n$ is a holonomic sequence. \\

To this end, let us express the homogeneous differential equation of 
$A(x)$ as
\[ p_0(x)A(x)+ p_1(x)A'(x)+ \dots + p_{r_1}(x)A^{(r_1)}(x) = 0,\]
and the homogeneous differential equation of 
$B(x)$  as
\[q_0(x)B(x)+ q_1(x)B'(x)+ \dots + q_{r_2}(x)B^{(r_2)}(x) = 0.\]

Note that $C(x)=A(x)\cdot B(x)$
and that any order derivatives of $C(x)$ can be
written as a linear combination of $A^{(i)}(x)\cdot B^{(j)}(x),$
\;\  $0 \leq i \leq r_1-1$ and $0 \leq j \leq r_2-1$. 
Hence, we can write the relation in matrix notation as: 
\footnotesize 
\[  \begin{bmatrix}
C(x)\\
C'(x)\\
C''(x)\\
\dots \\
C^{(r_1r_2-1)}(x) \\
C^{(r_1r_2)}(x) 
\end{bmatrix} =   
 \begin{bmatrix}
1 & 0 & 0 & \dots & 0 & 0 & 0 & \dots \\
0 & 1 & 0 & \dots & 1 & 0 & 0 & \dots \\
&&&\dots \\
M_{r_1r_2-1}^{(1)}(n) & M_{r_1r_2-1}^{(2)}(n) & M_{r_1r_2-1}^{(3)}(n) &   \dots &  &  &  &  \\
M_{r_1r_2}^{(1)}(n) & M_{r_1r_2}^{(2)}(n) & M_{r_1r_2}^{(3)}(n) &   \dots &  &  &  &  
\end{bmatrix}  \cdot
 \begin{bmatrix}
A(x)B(x)\\
A(x)B'(x)\\
\dots\\
A'(x)B(x)\\
A'(x)B'(x)\\
\dots\\
A^{(r_1-1)}(x)B^{(r_2-1)}(x)
\end{bmatrix}.\]
\normalsize

Using the same arguments as above, since the matrix $M$ has $r_1r_2+1$ rows and $r_1r_2$ columns, the existence of a non-trivial polynomial solution $P$ (in the null space of $M$) is ensured. This $P$ gives a homogeneous relation in terms 
of the differential equation of order $r_1r_2$ to $C(x), C'(x), \dots, C^{(r_1r_2)}(x).$  \qedhere
\end{itemize}
\end{proof}
\begin{cor}
Let $\displaystyle A(x) = \sum_{n=0}^{\infty} a_nx^n$ be holonomic. Then
$A'(x)$ and $\int A(x)$ are also holonomic.
\end{cor}

\begin{proof}
$\displaystyle A'(x) = \sum_{n=0}^{\infty}(n+1)a_{n+1}x^n.$ 
Since $n+1$ and $a_{n+1}$ are holonomic sequences, by the closure property of multiplication, so is the sequence $(n+1)a_{n+1}$. 
Hence, $A'(x)$ is  holonomic. A similar argument holds for $\int A(x).$
\end{proof}

\textbf{Example: Closure properties}
  
Let $ \displaystyle a_n =\dfrac{1}{n+1}\binom{2n}{n}$, the Catalan numbers. \;\
Recall that $a_n$ satisfies the holonomic recurrence of order 1 and degree 1: 
\[ (4n+2)a_n-(n+2)a_{n+1}=0.\]

Let $ \displaystyle b_n =\sum_{i=1}^n \dfrac{1}{i}$, the harmonic numbers. \;\
Here, $b_n$ satisfies the holonomic recurrence of order 2 and degree 1: 
\[ (n+1)b_n-(2n+3)b_{n+1}+(n+2)b_{n+2} =0.\]

We first show detailed calculations for the closure properties of addition and term-wise multiplication.

\begin{itemize}[leftmargin=0.2in]
\item addition: $c_n=a_n+b_n$. Consider a matrix equation:
\[  \begin{bmatrix}
c_n\\
c_{n+1}\\
c_{n+2}\\
c_{n+3}
\end{bmatrix} =   
 \begin{bmatrix}
1 & 1 & 0  \\
\frac{2(2n+1)}{n+2} & 0 & 1  \\
\frac{4(3+2n)(1+2n)}{(3+n)(2+n)} & -\frac{n+1}{n+2} & \frac{2n+3}{n+2}  \\
\frac{8(5+2n)(3+2n)(1+2n)}{(4+n)(3+n)(2+n)} & -\frac{(5+2n)(n+1)}{(3+n)(2+n)} 
& \frac{(12n+3n^2+11)}{(3+n)(2+n)}  
\end{bmatrix}  \cdot
 \begin{bmatrix}
a_n\\
b_{n}\\
b_{n+1}
\end{bmatrix}.\]

The polynomial solution $P$ in the null space of the rational matrix $M$ is
\[ P = [ -2(n+1)(3n+7)(1+2n)(2+n)^2,   (5+3n)(2+n)(9n^3+43n^2+58n+20), \dots ].   \]
This gives rise to the holonomic relation of $c_n$ of order 3 as
\begin{align*}
&-2(n+1)(3n+7)(1+2n)(2+n)^2c_n+(5+3n)(2+n)(9n^3+43n^2+58n+20)c_{n+1} \\
&-(216n+241n^2+111n^3+64+18n^4)(3+n)c_{n+2}\\
&+(3+n)(4+n)(3n+4)(n+1)^2c_{n+3} =0. 
\end{align*}

\begin{code}
> R1:= 2+4*n+(-2-n)*N:
> R2:= n+1+(-3-2*n)*N+(2+n)*N^2:
> HoAdd(R1,R2,n,N,c);
\end{code}




\item  term-wise multiplication: $d_n=a_n\cdot b_n$. Consider
\[  \begin{bmatrix}
d_n\\
d_{n+1}\\
d_{n+2}
\end{bmatrix} = \begin{bmatrix}
1 &  0  \\
 0 & \frac{2(2n+1)}{n+2}   \\
\frac{4(3+2n)(1+2n)(-1-n)}{(3+n)(2+n)^2} & \frac{4(3+2n)^2(1+2n)}{(3+n)(2+n)^2} 
\end{bmatrix}  \cdot
 \begin{bmatrix}
a_nb_n\\
a_nb_{n+1}
\end{bmatrix}.
\]

The polynomial solution $P$ in the null space of the rational matrix $M$ is
\[ P = [ 4(3+2n)(1+2n)(n+1),-2(3+2n)^2(2+n),(2+n)^2(3+n)].   \]
This gives rise to the holonomic relation of $d_n$ of order 2 as
\begin{align*}
&4(3+2n)(1+2n)(n+1)d_n-2(3+2n)^2(2+n)d_{n+1}
+(2+n)^2(3+n)d_{n+2} =0. 
\end{align*}

\begin{code}
> R1:= 2+4*n+(-2-n)*N:
> R2:= n+1+(-3-2*n)*N+(2+n)*N^2:
> HoTermWise(R1,R2,n,N,c);
\end{code}



We now give an example that shows how to obtain a homogeneous differential equation for the  Cauchy product.

\item Cauchy product: $e_n= \sum_{i=0}^n a_i\cdot b_{n-i}$. 

In this example, we let
$\displaystyle a_n = \dfrac{1}{n+1}\binom{2n}{n}$
and $b_n = n!.$

As $a_n$ and $b_n$ are holonomic, the 
homogeneous differential equations $A(x)$ and $B(x)$ exist. Indeed, they satisfy
\begin{align*}
x(4x-1)A''(x)+2(5x-1)A'(x)+2A(x) &= 0, \\
x^2B''(x)+(3x-1)B'(x)+B(x) &= 0.
\end{align*}

Letting $E(x)=A(x)\cdot B(x)$, we obtain the matrix equation
\[  \begin{bmatrix}
E(x)\\
E'(x)\\
E^{''}(x)\\
E^{(3)}(x)\\
E^{(4)}(x)
\end{bmatrix} = \begin{bmatrix}
1 &  0 &0 &0 \\
0 & 1 & 1 & 0 \\
-\dfrac{1}{(x^2)}-\dfrac{2}{x(4x-1)}& -\dfrac{(3x-1)}{x^2}
& -\dfrac{2(5x-1)}{x(4x-1)}& 2 \\
& \dots & & \\
& \dots & &
\end{bmatrix}  \cdot
 \begin{bmatrix}
A(x)B(x) \\
A(x)B'(x) \\
A'(x)B(x) \\
A'(x)B'(x) \\
\end{bmatrix}.
\]

The polynomial solution $P$ in the null space of the rational matrix $M$ is found to be
\[ P = [ 2(72x^6+660x^5-1392x^4+900x^3-266x^2+37x-2)(4x-1), \dots].   \]
This gives rise to the homogeneous differential equation 
of $C(x)$ of order $2\cdot2 = 4$ as
\begin{align*}
&2(72x^6+660x^5-1392x^4+900x^3-266x^2+37x-2)(4x-1)E(x) \\
&+2(1512x^8+11076x^7-26812x^6+22170x^5-9442x^4+2333x^3-342x^2+28x-1)E'(x) \\
&+\dots+x^5(4x-1)^2(4x^4+24x^3-31x^2+10x-1)E^{(4)}(x)=0. 
\end{align*}

\begin{code}
> DA:= lhs(HoToDiffHom(4*n+2-(n+2)*N,[1],n,N,x,D))
/f(x);
> DB:= lhs(HoToDiffHom(n+1-N,[1],n,N,x,D))/f(x);
> HoCauchy(DA,DB,x,D,c); 
\end{code}

\end{itemize}

\textbf{A remark on the rigorous proof of identities and the upper bound for the degree of recurrence relations} 

As already mentioned previously in the $C$-finite section, the closure property is utterly important. Theorem \ref{hoclose} gives bounds for the order of recurrence relation. Meanwhile, bounds for the degree of polynomial coefficients can also be determined. Once the bounds for the order and degree are known, if the holonomic ansatz is fit with enough data, a rigorous recurrence relation can be ensured. 

We now illustrate how we can obtain a good pragmatic upper bound for the degree using the following example. 

{\bf Example:} Let us try to figure out an upper bound of the degree of the recurrence relation under the addition: $c_n=a_n+b_n$. Using the same example as above, recall that the bound for the order of $c_n$ is $r=2+1=3$, and the following rational matrix was obtained in the intermediate steps:

\[  M =   
 \begin{bmatrix}
1 & 1 & 0  \\
\frac{2(2n+1)}{n+2} & 0 & 1  \\
\frac{4(3+2n)(1+2n)}{(3+n)(2+n)} & -\frac{n+1}{n+2} & \frac{2n+3}{n+2}  \\
\frac{8(5+2n)(3+2n)(1+2n)}{(4+n)(3+n)(2+n)} & -\frac{(5+2n)(n+1)}{(3+n)(2+n)} 
& \frac{(12n+3n^2+11)}{(3+n)(2+n)}  
\end{bmatrix}. \]

Getting rid of the denominator, we obtain the matrix 
\[  \tilde{M} =   
 \begin{bmatrix}
(4+n)(3+n)(2+n) & (3+n)(2+n) & 0  \\
2(2n+1)(4+n)(3+n) & 0 & (3+n)(2+n)  \\
4(3+2n)(1+2n)(4+n) &  -(n+1)(3+n)
& (2n+3)(3+n)\\ 
 8(5+2n)(3+2n)(1+2n)&  -(5+2n)(n+1)
& 12n+3n^2+11\\
\end{bmatrix} \] 
whose polynomial entries have the highest degree of $v=3$.

Finding a bound for the degree of $c_n$ amounts to determining the degree of the polynomial solution $P=\left[p_1(n), p_2(n), \cdots, p_{r+1}(n)\right]$ in the null space of $\tilde{M}$.

Formally, let us assume that the highest degree of polynomial in $P$ is $k$. Then, $\displaystyle p_l(n)= \sum_{j=0}^k b_{j,l}n^j$, $l=1,2,\dots,r+1$. By the method of equating the coefficients, the matrix equation $P\cdot \tilde{M}=0$ results in the system of $r(k+1+v)$  linear equations with $(r+1)(k+1)$ unknowns. The existence of a solution to this linear system is guaranteed whenever $k$ satisfies the inequality $(r+1)(k+1)>r(k+1+v)$, i.e. $k\geq rv$.
Therefore, this condition gives rise to a pragmatic upper bound for $k$, the degree of the holonomic sequence $c_n$.
\item  {\bf   Asymptotic approximation solutions:}

Unlike the $C$-finite recurrences, no closed form solution is available for holonomic recurrences. Hence, an approximation solution in terms of asymptotic expansion will be sought for a holonomic recurrence. As the method is quite complicated, we will not attempt to provide a theoretical analysis, but rather some applications of the method.
For a more thorough account on the subject, the interested reader is referred to \cite{JZ}.
In what follows, the sequence $a_n$ will be treated as a function. To emphasize this fact we will denote it by $y(n)$.\\

Suppose that $y(n)$ is a solution to 
\begin{equation}
\label{eq:yn}
    \sum_{i=0}^r p_i(n) y(n+i) = 0, 
\end{equation}
where  $p_r(n) \neq 0, \;\  n=0,1,2,\dots.$ 

The approach is based on the Birkhoff-Trjitzinsky method \cite{B, BT}. Although the detailed analysis of the method was given in \cite{B, BT}, we adopt here a variant which assumes a solution in its simplest form of \eqref{BT}. Then, substitute the assumed solution into 
\eqref{eq:yn} (with the help of \eqref{rat1}) to find values for the parameters.
For this reason,  we will refer to the method as the \emph{guess and check} in this paper. 
Despite its simplicity, this variant has proven to be applicable to a large class of holonomic recurrences.

\vspace{1em}

\textbf{Guess and check method:}

The guess and check method is a general method of solving holonomic recurrences. What we will do is to try a simple solution of the form
\begin{equation}  \label{BT} 
y(n) =  e^{\mu_0n\ln{n}+\mu_1n} \cdot n^{\theta} \cdot
e^{\alpha_1n^{\beta}+\alpha_2n^{\beta-\frac{1}{\rho}}
+\alpha_3n^{\beta-\frac{2}{\rho}}+\dots},
\end{equation}
where the parameters $\mu_0, \mu_1, \theta$ are complex, $\rho \, (\rho \geq 1), 
\mu_0 \rho$ are integers, 
and $\alpha_1 \neq 0,  \;\ \beta = \frac{j}{\rho}, \, 0 \leq j < \rho.$

This method provides $r$ independent solutions (all formal series solutions) but only up to constant multiple. The function in the form of \eqref{BT} which satisfies \eqref{eq:yn} is called \emph{a formal series solution} (FSS).

Using some power series expansion and simplification, 
we obtain
\begin{equation}  \label{rat1} 
\dfrac{y(n+k)}{y(n)} = n^{\mu_0k}\lambda^k\cdot\{ 1
+\dfrac{1}{n}\left(\theta k +\dfrac{k^2\mu_0}{2}\right)
+\dots \}\cdot 
e^{\alpha_1\beta k n^{\beta-1}
+\alpha_2(\beta-\frac{1}{\rho})kn^{\beta-\frac{1}{\rho}-1}+\dots},
\end{equation}
for $k\geq0$, where $\lambda = e^{\mu_0+\mu_1}$.

\textbf{Applications:}

Let us give walkthrough examples to demonstrate the approach. Since the procedure  involves some steps that require human input and expertise, no Maple program is provided in this section. \vspace{1em}

\textbf{Example 1:}  $y(n) = n!$. The most standard and widely used asymptotic formula for the factorial function is Stirling's formula. In this example, we will try to obtain an asymptotic approximation for the factorial function using  the method of guess and check.

From the recurrence relation  $y(n+1)-(n+1)y(n) = 0$, we apply \eqref{rat1}: 
\[ n^{\mu_0}\lambda\{1+\dfrac{1}{n}\left(\theta +\dfrac{\mu_0}{2}\right)+\dots\}
e^{\alpha_1\beta n^{\beta-1}+\dots}-(n+1) = 0. \]

Expanding the exponential term with power series and comparing the terms involving $n^{\mu_0}$ and $n$, we have 
$\mu_0=1, \, \lambda = 1$. Also, since $\mu_0 \rho$ must be an integer, 
we assign $\rho = 1$, the minimum possible value. 

Next, the value of $\beta$ must be 0, as the coefficient of $n^s$ for $0 < s < 1$ must be 0.

For the coefficient of 1 (the constant term), we have
$\theta+\dfrac{\mu_0}{2}-1=0$. Hence, $\theta = \dfrac{1}{2}.$

Plugging in these parameters back to \eqref{BT}, we arrive at
\[ y(n) = K\left(\dfrac{n}{e}\right)^n\sqrt{n}\left(1+\dfrac{c_1}{n}+\dfrac{c_2}{n^2}+\dots  \right).\]

We note that the infinite series on the right most arises from the series expansion of
$e^{\alpha_2n^{-1}+\alpha_3n^{-2}+\dots}.$
The values of $c_1, c_2, \dots$ can be figured out by plugging $y(n)$ back into the original recurrence and comparing the coefficient of $n^i$
for each $i$ (the method of undetermined coefficients).
The constant $K$, however, cannot be obtained by this method although 
another asymptotic method shows that $K= \sqrt{2\pi}.$

\textbf{Example 2:} 
Let $ \displaystyle y(n) = \left \lfloor{\left(\dfrac{n}{2}\right)^2}\right \rfloor$
with the holonomic recurrence
\[ (n+2)y(n)+2y(n+1)-ny(n+2) =0.\]

Apply \eqref{rat1}  to get the relation: 
\small
\[ (n+2)+2n^{\mu_0}\lambda\{1+\dfrac{1}{n}\left(\theta +\dfrac{\mu_0}{2}\right)+\dots\}
e^{\alpha_1\beta n^{\beta-1}+\dots}
-nn^{2\mu_0}\lambda^2\{1+\dfrac{1}{n}\left(2\theta +2\mu_0\right)+\dots\}
e^{2\alpha_1\beta n^{\beta-1}+\dots} = 0. \]
\normalsize

Expanding the exponential term with power series and comparing the terms involving $n^{2\mu_0}$ and $n$, we have $\mu_0=0$. Then, $1-\lambda^2=0$, which gives $\lambda = \pm 1$. Also, since $\mu_0 \rho$ must be an integer, 
we again assign the minimum possible value $\rho = 1$. 

Next, $\beta$ must be 0, as the coefficient of $n^s$ for $0 < s < 1$ must be 0.

For the coefficient of 1 (the constant term), we have
$2+2\lambda-\lambda^2(2\theta+2\mu_0)=0$. 
Hence, $\theta = \dfrac{1+\lambda}{\lambda^2} = 2$ or $0$.

Plugging these parameters back into \eqref{BT}, we obtain
\[ y_1(n) = K_1 n^{2}\left(1+\dfrac{c_1}{n}+\dfrac{c_2}{n^2}+\dots  \right),\]
\[ y_2(n) = K_2 (-1)^n\left(1+\dfrac{d_1}{n}+\dfrac{d_2}{n^2}+\dots  \right),\]
where $y(n) = y_1(n)+y_2(n)$.

With this form of solution, 
we apply the method of undetermined coefficients to the original recurrence, which in turn implies that $c_i$ and $d_i$ 
are all zero except $c_2=-1/2$. Hence, $y(n) = K_1n^2-\dfrac{K_1}{2}+K_2(-1)^n$, 
agreeing with our earlier result that 
$y(n) = \dfrac{n^2}{4}-\dfrac{1}{8}+\dfrac{(-1)^n}{8}.$
\end{enumerate}

\section{$C^2$-finite}
\label{sec:C2}

Building upon $C$-finite, the class of $C^2$-finite sequences that we will investigate in this section is a relatively new domain of research. The idea was first mentioned in \cite{KM} in the context of graph polynomials, and  it has recently been discussed 
in \cite{JNP,TZ} from a theoretical perspective.

\begin{enumerate}[leftmargin=0.4in,label=\textbf{\arabic*.}]
\item  {\bf   Ansatz:} $a_n$ is defined by a linear recurrence with $C$-finite sequence coefficients:
\[ C_{r,n}a_{n+r} + C_{r-1,n}a_{n+r-1}+ \dots +C_{0,n}a_{n} = 0, \]
where  $C_{r,n} \neq 0, \, n=0,1,2, \dots$,
along with the initial values $a_0, a_1, \dots, a_{r-1}.$
We call $a_n$ a \emph{$C^2$-finite sequence} of order $r$.
This term was first coined in \cite{KM}.

As was the case for holonomic sequences, the condition $C_{r,n} \neq 0$ is necessary for recursively computing the value of $a_{n+r}$ from preceding terms.

\item  {\bf   Example 1:} $C^2$-finite sequence of order $1$ given by
\[ a_n = F_{n+1} \cdot a_{n-1}, \qquad  a_0=1,\]
where $F_n$ is the Fibonacci sequence.

An interesting fact about this sequence is that it also satisfies
a non-linear relationship
\[ a_{n}a_{n+1}a_{n+3} - a_{n}a_{n+2}^2 - a_{n+2}a_{n+1}^2 = 0.  \]
This nonlinear relation was provided by Robert Israel/Michael Somos in 2014. 
It is the sequence A003266 on the Sloane’s OEIS, \cite{OEIS}.

This example motivated us to investigate connections between the class of $C^2$-finite sequences and a non-linear recurrence representation. Especially, since the conditions which ensure the existence of nonlinear recurrences 
for the $C$-finite (sub)sequences have already been examined in \cite{EZ}, we are curious to know if similar results might be obtained for $C^2$-finite sequences.
Furthermore, it is also unclear whether or not one can find conditions for which a non-linear recurrence is a $C^2$-finite sequence. We leave these as open problems for the interested readers. 

\textbf{Open problem 1:}
Find conditions which guarantee that a $C^2$-finite sequence can be represented in a non-linear recurrence relation.

\textbf{Open problem 2:}
Find conditions which guarantee that a non-linear recurrence is a $C^2$-finite sequence.

{\bf Example 2:} $C^2$-finite sequence of order $2$ given by
\[ a_{n+2} = a_{n+1}+2^na_{n}, \] 
with initial values $a_0=1, a_1=1.$
In terms of the shift operator,
\[  [N^2-N-2^n] \cdot a_n = 0.  \]

{\bf Example 3:} $C^2$-finite sequence of order $2$ given by
\[ a_{n+2} = F_{n+1} a_{n+1}+F_{n}a_{n}, \] 
with initial values $a_0=1, a_1=1.$

\item  {\bf   Guessing:} 

Input: the order $r$ of $a_n$, the order $d$ for each $C$-finite coefficient $C_{i,n},\, 0 \leq i \leq r$, and a sufficiently long sequence of data $a_n$.

This time guessing becomes very difficult due to the challenge we face when solving a system of nonlinear equations. We illustrate this by the following examples.

The simplest, non-trivial example is the second order relation
where $C_{1,n}$ and $C_{0,n}$ are of first order.
The form of ansatz is 
\[  a_{n+2} = c_1\alpha^na_{n+1}+c_2\beta^na_{n},  \;\ \;\ n \geq 0.   \] 
Here, we solve for constants $\{\alpha, \beta, c_1, c_2\}$ through
the system of nonlinear equations. 
With four parameters, Maple can still handle the computation in this case.

However, if the second order relation is assumed with $C_{1,n}$ and $C_{0,n}$ of second order, i.e.
\[  a_{n+2} = (c_1\alpha_1^n+c_2\alpha_2^n)a_{n+1}
+(c_3\alpha_3^n+c_4\alpha_4^n)a_{n},  \;\ \;\ n \geq 0,  \] 
with eight parameters to solve for, this time the problem becomes computationally infeasible. 

Another approach that could be useful for guessing a $C^2$-finite relation is to apply a numerical solution method. In Maple, we can do this with the available \texttt{fsolve} built-in command.
Unfortunately, even with the numerical method, we were not able to obtain the solution within a finite number of steps.
It was rather disappointing to find that guessing for $C^2$-finite is not practical, as it plays a big role in determining an expression for the sequences.

\item  {\bf   Generating function:} 

In this section, we establish several new properties for $C^2$-finite sequences. First, we give a formal definition of $C^2$-finite. 


We recall from the $C$-finite section that 
a closed-form formula for a $C$-finite sequence $C_n$ is 
\[ C_n = \sum_{\alpha \in S} p_{\alpha}(n)\alpha^n,   \]
where $\alpha$'s are the roots of the characteristic polynomial
of $C_n$.

We define $Deg(C_n)$ to be the highest degree of
$p_{\alpha}(n), \, \alpha \in S.$

\begin{defi}
A $C^2$-finite sequence $a_n$ is said to have order $r$ and
degree $k$ if $a_n$ satisfies the recurrence relation 
\[ C_{r,n}a_{n+r} + C_{r-1,n}a_{n+r-1}+ \dots +C_{0,n}a_{n} = 0, \]
where for each $i, \, 0 \leq i \leq r, \;\ Deg(C_{i,n})$ is at most $k.$
\end{defi}

We are now ready to derive a new differential equation for the generating function of $C^2$-finite sequence. This inquiry was made in \cite{JNP}.

\begin{thm}  \label{GenC2}
Let $\displaystyle f(x) = \sum_{n=0}^{\infty} a_nx^n $
where $a_n$ is a $C^2$-finite sequence of order $r$ and 
degree $k$.
Then, $f(x)$ satisfies a (non-homogeneous) 
linear differential equation with polynomial coefficients,
\begin{equation}    \label{diffcom}
\sum_{\alpha} \left[ q_{\alpha,0}(x)f(\alpha x)+ q_{\alpha,1}(x)f'(\alpha x)
+ \dots + q_{\alpha,r'}(x)f^{(r')}(\alpha x)\right] = R(x)  
\end{equation}
where $\alpha$ is defined to be $\alpha_{i,j}$, the root of the characteristic polynomial of $C_{i,n}.$
Here, the order $r'$ is at most $k$, degree of $q_{\alpha,t}(x)$ 
for each $\alpha$ and $t$
is at most $r+k$, and the degree of polynomial $R(x)$ is at most $r-1.$
\end{thm}

{\bf Notation.}  To avoid any ambiguity in our notation $f^{(j)}(\alpha x)$, this notation means 
\[
f^{(j)}(\alpha x)=\dfrac{d^jf(\alpha x)}{dx^j}.
\]

\begin{defi}
The function $f(x)$ as a generating function
of a $C^2$-finite sequence is called \emph{$C^2$-finite}.
Also,  we will call a function $f(x)$ that satisfies \eqref{diffcom} \emph{$DC$-finite} (differentiably composite finite). 
\end{defi}

\begin{rem}
In contrast with the $DC$-finite, another approach to generalizing the holonomic sequences is through the $D$-finite generating function (of a holonomic sequence). This was considered in \cite{JP} and the resulting generating function is known as \emph{$DD$-finite}.
\end{rem}

\begin{proof}
Following the idea of the proof of Theorem \ref{GenHolo}, assume that $a_n$ satisfies the relation   
\[  
C_{r,n}(n)a_{n+r} + C_{r-1,n}(n)a_{n+r-1}+ \dots +C_{0,n}(n)a_{n} = 0,
\]
where $C_{t,n}(n)$  can be written in a closed-form as 
$\displaystyle C_{t,n}(n) = \sum_{\alpha \in S_t} p_{t,\alpha}(n)\alpha^n.$

We denote by $b_{s,t,\alpha}$ the coefficient of $n^s\alpha^na_{n+t}$ in the relation above. Then, $\displaystyle p_{t,\alpha}(n)=\sum_{s=0}^k b_{s,t,\alpha}n^s$, and so the $C^2$-finite relation becomes

\begin{equation}  \label{c2rela}
\sum_{t=0}^r\sum_{\alpha \in S_t}\sum_{s=0}^k b_{s,t,\alpha}n^s\alpha^n a_{n+t}=0.
\end{equation}

We next prove the following identity. For fixed $s$ and $t$,

\begin{align*}
\sum_{j=0}^s c^{(s,t)}_jx^{j-t}f^{(j)}(\alpha x) 
&=  \sum_{j=0}^s c^{(s,t)}_j  \sum_{n=0}^{\infty} (n)_ja_n \alpha^n  x^{n-t} \\
&=  \sum_{j=0}^s c^{(s,t)}_j  \sum_{n=-t}^{\infty} (n+t)_ja_{n+t}\alpha^{n+t}x^{n} \\
&= \alpha^t  \sum_{n=-t}^{\infty} \left[ \sum_{j=0}^s c^{(s,t)}_j (n+t)_j \right] a_{n+t}
(\alpha x)^{n}.
\end{align*}

For each pair of $(s,t)$, we solve for constants 
$c^{(s,t)}_j, \;\ j=0,1,2,\dots,s, $ by equating coefficients of $n^j$ in the equation  
$\displaystyle \sum_{j=0}^s c^{(s,t)}_j (n+t)_j = n^s$. Now,  define $\displaystyle A_{s,t,\alpha}(x) = \sum_{n=0}^{\infty} a_{n+t}n^s (\alpha x)^n$ for fixed $s,t \geq 0$. Then,
\begin{equation} \label{rank} 
\sum_{j=0}^s c^{(s,t)}_jx^{j-t}f^{(j)}(\alpha x) 
=  \alpha^t \sum_{n=-t}^{\infty} a_{n+t}n^s (\alpha x)^{n}
= \alpha^t A_{s,t,\alpha}(x) + \alpha^t \sum_{n=-t}^{-1} a_{n+t}n^s (\alpha x)^{n}.
\end{equation}

From the $C^2$-finite relation of $a_n$ in \eqref{c2rela}, 
multiply $x^n$ through, and sum $n$ from $0$ to $\infty$, 
we obtain
\begin{align*} 
0 &= \sum_{n=0}^\infty\sum_{t=0}^r\sum_{\alpha \in S_t}\sum_{s=0}^k b_{s,t,\alpha}n^s\alpha^n a_{n+t}x^n= \sum_{t=0}^r\sum_{s=0}^k\sum_{\alpha \in S_t} 
b_{s,t,\alpha}A_{s,t,\alpha}(x)   \\
&= \sum_{t=0}^r\sum_{s=0}^k\sum_{\alpha \in S_t} b_{s,t,\alpha} 
\left[\alpha^{-t}\sum_{j=0}^s c^{(s,t)}_jx^{j-t}f^{(j)}(\alpha x)
-  \sum_{n=-t}^{-1} a_{n+t}n^s(\alpha x)^{n} \right] , \;\  \text{ from \eqref{rank}.}
\end{align*}
Multiply $x^r$ on both sides and rearrange this equation:
\[    \sum_{t=0}^r\sum_{s=0}^k\sum_{\alpha \in S_t} b_{s,t,\alpha} \alpha^{-t}
\sum_{j=0}^s c^{(s,t)}_jx^{j+r-t}f^{(j)}(\alpha x)
= \sum_{t=0}^r\sum_{s=0}^k\sum_{\alpha \in S_t} b_{s,t,\alpha} 
\sum_{n=0}^{t-1} a_{n}(n-t)^s \alpha^{n-t}  x^{n+r-t} . \]
We see that the left hand side is the differential equation of order at most $k$ and degree at most $r+k$. The right hand side is the polynomial of degree at most $r-1.$
\end{proof}

{\bf Example 1:} Let $a_{n+1} = F_{n+2} \cdot a_{n}.$ Then,  
\begin{align*}
f(x) &= \sum_{n=0}^{\infty} a_nx^n 
=  a_0+\sum_{n=1}^{\infty} F_{n+1} \cdot a_{n-1}x^n
= a_0+x\sum_{n=1}^{\infty} (c_1\alpha_+^{n+1}+c_2\alpha_-^{n+1}) \cdot a_{n-1}x^{n-1}  \\
&= a_0+c_1\alpha_+^2x\sum_{n=0}^{\infty} \alpha_+^n \cdot a_nx^{n}
+c_2\alpha_-^2 x\sum_{n=0}^{\infty}\alpha_-^n \cdot a_nx^{n} , \;\ \mbox{ (shift index $n$ by 1)}\\
&= a_0+  c_1\alpha_+^2xf(\alpha_+x)+ c_2\alpha_-^2xf(\alpha_-x),
\end{align*}
where $\alpha_+$ and $\alpha_-$ are the roots of equation $x^2-x-1=0.$

\begin{code}
> C2ToDiff(N-(c1*a^(n+2)+c2*b^(n+2)),{1,a,b},[a0],
n,N,x,D);
\end{code}



We still consider a first-order relation in the next example, but this time the coefficient $C_n$  has a polynomial factor.

{\bf Example 2:} Let $a_{n+1} = (n+1)2^n \cdot a_{n}.$ Then, 
\begin{align*}
f(x) &= \sum_{n=0}^{\infty} a_nx^n 
= a_0 +\sum_{n=1}^{\infty} n2^{n-1} \cdot a_{n-1}x^n
= a_0 +x\sum_{n=0}^{\infty}  (n+1)2^n \cdot a_nx^{n} \\
&= a_0 + x^2\sum_{n=1}^{\infty} n2^n \cdot a_nx^{n-1}
+x\sum_{n=0}^{\infty} 2^n \cdot a_nx^{n} \\
&=   a_0 + x^2f'(2x)+ xf(2x).
\end{align*}

\begin{code}
> C2ToDiff(N-(n+1)*2^n,{1,2},[a0],n,N,x,D);
\end{code}



Let us now consider a second-order example.

{\bf Example 3:} Let $a_{n+2} = a_{n+1} + 2^n\cdot a_{n}.$ Then 
\begin{align*}
f(x) &= \sum_{n=0}^{\infty} a_nx^n 
= a_0 + a_1x+\sum_{n=2}^{\infty} a_{n-1}x^n
+\sum_{n=2}^{\infty} 2^{n-2} \cdot a_{n-2}x^n \\
&= a_0 + a_1x-a_0x+x\sum_{n=0}^{\infty} a_{n}x^n
+x^2\sum_{n=0}^{\infty}a_{n}(2x)^n \\
&=   a_0 + (a_1-a_0)x + xf(x)+ x^2f(2x).
\end{align*}

\begin{code}
> C2ToDiff(N^2-N-2^n,{1,2},[a0,a1],n,N,x,D);
\end{code}




Similar to holonomic sequences, the differential equation \eqref{diffcom} 
can be made homogeneous by differentiating multiple times until $R(x)$ becomes zero. 

\begin{cor}
Let $\displaystyle f(x) = \sum_{n=0}^{\infty} a_nx^n $
where $a_n$ is a $C^2$-finite sequence of order $r$ and degree $k$.
Then, $f(x)$ satisfies a homogeneous 
linear differential equation with polynomial coefficients
\begin{equation}   
\sum_{\alpha} \left[ q_{\alpha,0}(x)f(\alpha x)+ q_{\alpha,1}(x)f'(\alpha x)
+ \dots + q_{\alpha,r'}(x)f^{(r')}(\alpha x)\right] = 0, 
\end{equation}
where order $r'$ is at most $r+k$, and degree of $q_{\alpha,t}(x)$ for each $\alpha, t$ is at most $r+k$.
\end{cor}

\textbf{Example:} The homogeneous differential equation of $\displaystyle f(x) = \sum_{n=0}^{\infty} a_n x^n$ where
$a_{n+1}=F_{n+2} \cdot a_n$ is
\[f'(x)-c_1\alpha_{+}^2[f(\alpha_+x)+xf'(\alpha_+x)]-c_2\alpha_{-}^2[f(\alpha_-x)+xf'(\alpha_-x)]=0.\]

\begin{code}
> C2ToDiffHom(N-(c1*a^(n+2)+c2*b^(n+2)),{1,a,b},[a0],
n,N,x,D);
\end{code}

The next theorem ensures that we can always find a $C^2$-finite recurrence
relation for the coefficients $a_n$ of $f(x)$ which satisfies a homogeneous linear 
differential equation of composite variables with polynomial coefficients. 

\begin{thm}
\label{thm:conv_c2finite}
Let $\displaystyle f(x) = \sum_{n=0}^{\infty} a_nx^n$.
Assume $f(x)$ satisfies a homogeneous 
linear differential equation with polynomial coefficients
of order $r$ and degree $k$
\begin{equation}    \label{diffc2}
\sum_{\alpha} \left[ q_{\alpha,0}(x)f(\alpha x)+ q_{\alpha,1}(x)f'(\alpha x)
+ \dots + q_{\alpha,r}(x)f^{(r)}(\alpha x)\right] = 0. 
\end{equation}
Then, $a_n$ is a $C^2$-finite sequence of order at most $r+k$ and degree 
at most $r$.
\end{thm}

\begin{proof}
Assume that $\displaystyle q_{\alpha,t}(x) = \sum_{s=0}^kb_{\alpha,s,t}x^s.$
Then, for each $s, t, \alpha,$
\[ x^sf^{(t)}(\alpha x) = \sum_{n=t}^{\infty} (n)_ta_n\alpha^n x^{n-t+s}
= \alpha^{t-s} \sum_{n=s}^{\infty} (n+t-s)_ta_{n+t-s} (\alpha x)^{n} . \]

Now \eqref{diffc2} becomes
\[ \sum_{\alpha} \left[ \sum_{t=0}^r \sum_{s=0}^{k} b_{\alpha,s,t} \alpha^{t-s} 
\sum_{n=s}^{\infty}  (n+t-s)_ta_{n+t-s}(\alpha x)^{n} \right] = 0, \]
and so for each $n \geq 0, \;\ a_n$ satisfies the recurrence
\[  \sum_{t=0}^r \sum_{\alpha} \sum_{s=0}^{\min(n,k)} 
\left[b_{\alpha,s,t}(n+t-s)_t\alpha^{n+t-s}\right] a_{n+t-s} = 0. \]
Proposition \ref{prop:converse_closeform_cfinite} implies that the coefficient 
 $\displaystyle C_{i,n} = \sum_{\alpha} \sum_{s=0}^{\min(n,k)} b_{\alpha,s,s+i}(n+i)_{s+i}\alpha^{n+i}$, 
for each $i$, is a $C$-finite sequence. It follows that $a_n$ satisfying 
\[ C_{r,n}a_{n+r} + C_{r-1,n}a_{n+r-1}+ \dots +C_{-k,n}a_{n-k} = 0, \]
is a $C^2$-finite sequence.
The claim of the order and degree is now immediate.
\end{proof}

\item  {\bf   Closure properties:} 

\begin{thm}  \label{c2close}
Assume $a_n$ and $b_n$ are $C^2$-finite sequences.
The following are also $C^2$-finite sequences.
\begin{enumerate}[label=(\roman*)]
\item addition: $(a_n+b_n)_{n=0}^{\infty},$ 
\item term-wise multiplication: $(a_n \cdot b_n)_{n=0}^{\infty},$ 
\item Cauchy product: $(\sum_{i=0}^n a_i \cdot b_{n-i})_{n=0}^{\infty},$ 
\item partial sum: $(\sum_{i=0}^n a_i)_{n=0}^{\infty}$,  
\item linear subsequence: $(a_{mn})_{n=0}^{\infty}$, where $m$ is a positive integer.
\end{enumerate}
\end{thm}

The proof is along the same line as that of Theorem \ref{hoclose} for holonomic sequences, and we shall not repeat it here. This section will be devoted to a detailed discussion and examples instead.

\begin{rem}
The reader may have noticed that  this time we have not specified the upper bound of the order in the theorem. While the same bounds as those used for holonomic sequences could be imposed, it is worth repeating here that for a $C^2$-finite sequence of order $r$, the leading coefficient $C_{r,n}$ must not be zero for any $n \geq 0$.
This condition makes it not straightforward to determine a general bound for the order of the sequence.
The first example below (from \cite{JNP}) illustrates this issue.
\end{rem}

{\bf Example 1:} Let $a_n$ and $b_n$ be a $C^2$-finite sequence of order 1 defined by
\begin{align*}
a_{n+1}+(-1)^na_n &= 0, \\
b_{n+1}+b_n  &=0 .
\end{align*}
Let $c_n = a_n+b_n$ for $n \geq 0.$
A recurrence of order $2$ for $c_n$ is in the form
\[   [1-(-1)^n]c_{n+2} +2c_{n+1}+[1+(-1)^n]c_n=0. \]
This recurrence does not satisfy the definition of $C^2$-finite as
the leading term, $C_{2,n} = 1-(-1)^n$, contains infinitely many zeros.

On the other hand, a recurrence of order 3 makes $c_n$ a $C^2$-finite sequence:

\[ 
c_{n+3} + \frac{1}{2} \left[1+(-1)^n\right]c_{n+2}
+\frac{1}{2} \left[1-(-1)^n\right]c_n=0. 
\]

The following example illustrates the idea behind the derived recurrence relations under the addition and term-wise multiplication operations.

{\bf Example 2:} Let $a_n$ be a sequence that satisfies the relation 
\[ a_{n+1} = F_{n+2} \cdot a_{n}, \qquad  a_0=1,\]
where $F_n$ is the Fibonacci sequence.
Let $b_n$ be a sequence that satisfies the relation 
\[ b_{n+2} = b_{n+1}+2^nb_{n}, \;\ b_0=1, b_1=1. \] 

\begin{itemize}[leftmargin=0.2in]
    \item addition: $c_n = a_n+b_n.$ 

To solve for the recurrence relation
of $c_n$, we write $c_n, c_{n+1}, c_{n+2}$ and $c_{n+3}$ 
as a linear combination of $a_n, b_n$ and $b_{n+1}.$ That is,

\[  \begin{bmatrix}
c_n\\
c_{n+1}\\
c_{n+2}\\
c_{n+3}
\end{bmatrix} =   
 \begin{bmatrix}
1 & 1 & 0  \\
F_{n+2} & 0 & 1  \\
F_{n+3}F_{n+2} & 2^n & 1  \\
F_{n+4}F_{n+3}F_{n+2} & 2^n
& 2^{n+1}+1  
\end{bmatrix}  \cdot
 \begin{bmatrix}
a_n\\
b_{n}\\
b_{n+1}
\end{bmatrix}.\]

The $C$-finite solution $P$ (in the null space) of the $C$-finite sequence matrix $M$ is
\[ P = [ 2^nF_{n+2}(F_{n+4}F_{n+3}-F_{n+3}-2^{n+1}), \dots ],   \]
which gives rise to a $C^2$-finite relation of $c_n, \, n \geq 1$, of order 3
\begin{align*}
& 2^nF_{n+2}(F_{n+4}F_{n+3}-F_{n+3}-2^{n+1})c_n\\
&+\left[ F_{n+4}F_{n+3}F_{n+2}+2^{2n+1}
-2^{n+1}F_{n+3}F_{n+2}-F_{n+3}F_{n+2} \right]c_{n+1} \\
& + \left[ 2^{n+1}F_{n+2}+F_{n+2}
-F_{n+4}F_{n+3}F_{n+2}+2^n \right] c_{n+2}\\
&+\left[ F_{n+3}F_{n+2}-F_{n+2}-2^n \right] c_{n+3} =0. 
\end{align*}

\item term-wise multiplication: $d_n = a_n \cdot b_n.$ 

To solve for the recurrence relation
of $d_n$, we write $d_n, d_{n+1}$ and $d_{n+2}$ 
as a linear relation of $a_nb_n$ and $a_nb_{n+1}.$ That is,

\[  \begin{bmatrix}
c_n\\
c_{n+1}\\
c_{n+2}
\end{bmatrix} =   
 \begin{bmatrix}
1 &  0  \\
0 & F_{n+2}  \\
2^nF_{n+3}F_{n+2}
& F_{n+3}F_{n+2}
\end{bmatrix}  \cdot
 \begin{bmatrix}
a_nb_n\\
a_nb_{n+1}
\end{bmatrix}.\]

The $C$-finite solution $P$ (in the null space) of the $C$-finite sequence matrix $M$ is
\[ P = [ -2^nF_{n+2}F_{n+3}, -F_{n+3},1 ],   \]
yielding a $C^2$-finite relation of $c_n, \, n \geq 0$, of order 2:
\begin{align*}
-2^nF_{n+2}F_{n+3}c_n -F_{n+3}c_{n+1}+c_{n+2} = 0.
\end{align*}

\end{itemize}

As we mentioned earlier, the proof of closure properties for $C^2$-finite is similar to the holonomic one. It turns out that we can directly use the same Maple code to get recurrence relations for $C^2$-finite:

\begin{code}
> HoAdd(N-F(n+2),N^2-N-2^n,n,N,c);
> HoTermWise(N-F(n+2),N^2-N-2^n,n,N,c);
\end{code}

\item  {\bf   Asymptotic approximation solutions:}

We have already seen that the rate of growth of $C$-finite 
is $\mathcal{O}(\alpha^n)$ for some constant $\alpha$. 
Here, the rate of growth of $C^2$-finite is 
$\mathcal{O}(\alpha^{n^2})$ for some constant $\alpha$.  
Also, we have presented a procedure to obtain an asymptotic approximation solution for the holonomic recurrence in the previous section.
It appears, however, to be more difficult to derive asymptotic approximation solutions for the $C^2$-finite recurrences, and merits further investigation. 
We leave this as an open problem.

\textbf{Open problem 3:}
Find asymptotic approximations of solutions to the $C^2$-finite sequences. 

\end{enumerate}





\end{document}